\documentclass{ps}
\usepackage{amsmath,amsfonts}
\usepackage{color}
\usepackage[applemac]{inputenc}
\usepackage{mathrsfs}
\usepackage{mathtools}
\usepackage{hyperref}

\renewcommand{\H}{{\textbf{[H]}}}

\newcommand{\R}{{\mathbb R}}

\newcommand{\N}{{\mathbb N}}

\newcommand{\E}{{\mathbb E}}
\newcommand{\PP}{{\mathbb P}}

\newcommand{\ind}{{\textbf{1}}}

\theoremstyle{plain}
\newtheorem{theorem}{Theorem}[section]
\newtheorem{lemma}[theorem]{Lemma}
\newtheorem{proposition}[theorem]{Proposition}

\theoremstyle{definition}
\newtheorem{remark}{Remark}[section]

\begin{document}

\title{Density Estimates for SDEs Driven by Tempered Stable Processes}

\author{L. Huang}\address{Higher School of Economics, Moscow, lhuang@hse.ru}

\date{\today}

\begin{abstract}

We study a class of stochastic differential equations driven by a possibly tempered L\'evy process, under mild conditions on the coefficients.
We prove the well-posedness of the associated martingale problem as well as the existence of the density of the solution.
Two sided heat kernel estimates are given as well. Our approach is based on the Parametrix series expansion
\end{abstract}



\maketitle

\section{Introduction}

This paper is devoted to the study of Stochastic Differential Equations (SDEs), driven by a class of possibly tempered L\'evy processes.
Specifically, we show the existence of the density, as well as some associated estimates, under mild assumptions on the coefficients.
Weak uniqueness is also derived as a by-product of our approach.
More precisely, we study equations with the dynamics:
\begin{equation}\label{EDS1}
X_t = x + \int_0^t F(u,X_u) du + \int_0^t \sigma(u,X_{u^-}) dZ_u,
\end{equation}
where $F: \R_+\times \R^d \rightarrow \R^d$ is Lipschitz continuous, 
$\sigma: \R_+\times \R^d \rightarrow \R^d\otimes\R^d$ is measurable bounded, H\"older continuous in space and elliptic, and $(Z_t)_{t\ge 0}$ is a symmetric L\'evy process.
We will denote by $\nu$ its L\'evy measure and assume that it satisfies what we call \textit{a tempered stable domination:}
\begin{equation}\label{PREM_TEMP}
\nu(A) \le \int_{S^{d-1}} \int_0^{+\infty} \ind_{A}(s\theta) \frac{\bar q(s)}{s^{1+\alpha}}ds \mu(d\theta), \ \forall A \in \mathcal{B}(\R^d),
\end{equation}
where $\bar{q}$ is a non increasing function, and $\mu$ is a bounded measure on the sphere $S^{d-1}$. Also $ \mathcal{B}(\R^d)$ denotes the Borelians of $\R^d$.
This is a relatively large class of L\'evy processes, that contains in particular the stable processes.

%

In order to give density estimates on the solution of \eqref{EDS1}, it is first necessary to obtain density estimates for the driving process.
Those estimates are clear when $(Z_t)_{t\ge 0}$ is a Brownian motion.
However, the L\'evy case is much more complicated due to the huge diversity in the class of L\'evy processes.
Let us mention the papers of Bogdan and Sztonyk \cite{bogd:szto:07} and Kaleta and Sztonyk \cite{kale:szto:13} for density bounds concerning relatively general L\'evy processes.
In the case of the symmetric stable processes, the L\'evy measure writes:
\begin{equation}\label{STABLE_LEVY_MES}
\forall A \in \mathcal{B}(\R^d), \ \nu(A) = \int_0^{+\infty} \int_{S^{d-1}} \ind_{\{s\theta\in A\}} C_{\alpha,d} \frac{ds}{s^{1+\alpha}} \mu(d\theta),
\end{equation}
for some $\alpha \in (0,2)$.
In the above, $C_{\alpha,d}$ is a positive constant that only depends on $d$ and $\alpha$ (see Sato \cite{sato} for its exact value), and $S^{d-1}$ stands for the unit sphere of $\R^d$. Also, $\mu$ is a symmetric finite measure on the sphere called the spectral measure.
When the spectral measure satisfies the non-degeneracy condition:
\begin{equation}\label{ND_SPECT_MES}
\exists C >1, \ \mbox{ s.t. } \ C^{-1}|p|^\alpha \le \int_{S^{d-1}} |\langle p, \xi \rangle|^\alpha \mu(d\xi) \le C |p|^\alpha, 
\end{equation}
the driving process $Z_t$ has a density with respect to the Lebesgue measure.
In the recent work of Watanabe \cite{wata:07}, the author studied asymptotics for the density of a general stable process, and highlighted the importance of the spectral measure on the decay of the densities.
Specifically, let us denote by $p_Z(t,\cdot)$ the density of $Z_t$, and assume that there exists $\gamma>0$ such that
\begin{equation}\label{CONCENTRATION_SPECT_MES}
\mu \left(B(\theta,r) \cap S^{d-1} \right) \le C r^{\gamma -1}, \ \forall \theta \in S^{d-1}, \ \forall r\le 1/2,\ C\ge 1.
\end{equation}
Observe that in the case where the spectral measure has a bounded density with respect to the Lebesgue measure on $S^{d-1}$, this condition is satisfied with $\gamma =d$.
For a general $\gamma\in[1,d]$ such that \eqref{CONCENTRATION_SPECT_MES} holds, we have for all $ x \in \R^d$, $ t >0$:
\begin{equation}\label{UPPER_BOUND_WATA}
p_Z(t,x) \le C \frac{ t^{-d/\alpha}}{\left(1+\frac{|x|}{t^{1/\alpha}} \right)^{\alpha+\gamma}}.
\end{equation}
Moreover, a similar lower bound is given for the points $x\in\R^d$ such that two sided estimate hold in \eqref{CONCENTRATION_SPECT_MES} for $\theta= x/|x|$ (up to a modification of the threshold $r$).
We refer to Theorem 1.1 in Watanabe \cite{wata:07} for a thorough discussion. 
%
We would like to point out the difference between assumptions \eqref{CONCENTRATION_SPECT_MES} and \eqref{ND_SPECT_MES}.
The assumption \eqref{ND_SPECT_MES} alone is enough to show the existence of the density of the driving stable process.
However, it turns out that this sole assumption is not enough to get density estimates.
Instead, we need to know what we refer to (with a slight abuse of language) as the "concentration properties" of the spectral measure to deduce density bounds.
This concentration, 
reflected by the index $\gamma$ in \eqref{CONCENTRATION_SPECT_MES}, directly impacts the decay of the density, as shown in the bound \eqref{UPPER_BOUND_WATA}.
Observe however that if the concentration index $\gamma$ is too small with respect to the dimension, namely, $\alpha +\gamma\le d$, the upper bound \eqref{UPPER_BOUND_WATA} is not homogeneous to a density, since its integral (over $\R^d$) is not defined.
We refer to the work of Watanabe \cite{wata:07} for a detailed presentation of these aspects.

A generalization of this result to the case where the L\'evy measure does not factorize as in \eqref{STABLE_LEVY_MES}, but only satisfies the domination \eqref{PREM_TEMP} has been obtained by Sztonyk \cite{szto:10}.
%
Two sided estimates of the form \eqref{UPPER_BOUND_WATA} 
are derived, up to additional multiplicative terms involving the temperation $\bar{q}$ in \eqref{PREM_TEMP}, 
with the same restrictions for the lower bound.

The temperation $\bar{q}$ can be seen as a way to impose finiteness of the moments of $Z$ (see Theorem 25.3 in Sato \cite{sato}), and intuitively, the integrability properties of $(Z_t)_{t\ge 0}$ should transfer to $(X_t)_{t\ge 0}$.
However, giving a density estimate on the driving process and passing it to the density of the solution of the SDE is not always possible.

In the Brownian setting, if $\sigma \sigma^*$ is uniformly elliptic, bounded and H\"older continuous, 
and $F$ is Borel bounded, it is known that two sided Gaussian estimates hold for the density of the SDE \eqref{EDS1}.
This can be derived from the works of Friedman \cite{frie:64}.
We also mention the approach of Sheu \cite{sheu:91}, that also gives estimates on the logarithmic gradient of the density.
In the stable non degenerate case, i.e. when the coefficients $F,\sigma$ are as above, and $\mu(d\xi)$ has a smooth strictly positive density with respect to the Lebesgue measure on the sphere, it can be derived from Kolokoltsov \cite{kolo:00}, that the density $p(t,s,x,y)$ of \eqref{EDS1} exists and satisfies the following two sided estimates. Fix $T>0$, there exists $C>1$ depending on $T$, the coefficients and on the non degeneracy conditions, such that for all $x,y \in\R^d$, $0< t\le s\le T$:
\begin{equation}\label{ARONSON_STABLE_NON_DEG}
 C^{-1}\frac{ (s-t)^{-d/\alpha}}{\left(1+\frac{|x-y|}{(s-t)^{1/\alpha}} \right)^{\alpha+d}} \le p(t,s,x,y) \le C\frac{ (s-t)^{-d/\alpha}}{\left(1+\frac{|x-y|}{(s-t)^{1/\alpha}} \right)^{\alpha+d}}.
\end{equation}
 
This estimate is obtained using a continuity method: the parametrix technique.
This approach is well suited to obtain density estimates for the solution of an SDE under mild assumptions on the coefficients, provided that good estimates are available on the driving process and on the so-called \textit{parametrix kernel}.

We refer to estimates of the form \eqref{ARONSON_STABLE_NON_DEG} as Aronson estimates: two sided bounds that reflect the nature of the noise of the system.
In the Gaussian setting, the density of the solution has a Gaussian behavior, and in the stable case, the density of the solution has two sided bounds homogeneous to those of the driving stable process.
This work aims at proving Aronson estimates when the driving process is a L\'evy process satisfying a tempered domination (in the sense of \eqref{PREM_TEMP}).

This research domain has a long history in the literature and is intensely developing.
We can refer to Kochubei \cite{koch:89} or the book of Eidelman, Ivasyshen and Kochubei \cite{eide:Ivas:koch:04} in the parabolic setting. In the latter, the parametrix technique is presented as well.
Other techniques to obtain density estimates comes from the Dirichlet froms techniques.
See e.g. the works of Chen and Zhang \cite{chen:zhan:13}.
Our work differs from their in the fact that we can handle very general spectral measure with our approach.
We also mention the recent works of Knopova and Kulik \cite{knop:kuli:14} that uses the parametrix technique for stable driven SDEs, when the jump part (corresponding to our coefficient $\sigma$) is real valued.
This assumption greatly simplifies the proof of the upper bound for the parametrix kernel.
Also, they consider rotationally invariant stable process, which allows them to handle the case of a drift when $\alpha \le 1$.
Some recent development around the parametrix in Bally and Kohatsu-higa \cite{ball:koha:14} in the Brownian setting give a probabilistic interpretation of the parametrix technique.

Finally, we mention that existence of the density can be investigated via Malliavin calculus.
In the Brownian setting, we refer to the works of Kusuoka and Stroock \cite{kusu:stro:84,kusu:stro:85,kusu:stro:87}, as well as Norris \cite{norr:86}.
The jump case is more difficult, and is treated by various authors. Let us mention Bichteler, Gravereaux and Jacod \cite{bich:grav:jaco:87}, and Picard \cite{pica:96}.
However, this technique requires regularity on the coefficients.
In our approach, the convergence of the parametrix series will give us the existence of the density as well as weak uniqueness, under relatively mild assumptions on the coefficients.

We will denote by $\textbf{[H]}$ the following set of assumptions.
These hypotheses ensure the existence of the density, and are those required by Sztonyk \cite{szto:10} in order to have a two sided estimate for the driving process $Z$.
\begin{trivlist}
\item[\textbf{[H-1]}] $(Z_t)_{t\ge 0}$ is a symmetric L\'evy process without Gaussian part.
We denote by $\nu$ its L\'evy measure.
There is a non increasing function $\bar{q}: \R_+ \rightarrow \R_+$, $\mu$ a bounded measure on $S^{d-1}$, and $\alpha\in(0,2)$, $\gamma\in [1,d]$ such that:
\begin{eqnarray}
\nu(A) \le \int_{S^{d-1}} \int_0^{+\infty} \ind_{A}(s\theta) \frac{\bar q(s)}{s^{1+\alpha}}ds \mu(d\theta). 
\end{eqnarray}
We assume one of the following:
\begin{itemize}
\item[\textbf{[H-1a]}] $\mu$ has a density with respect to the Lebesgue measure on the sphere.
\item[\textbf{[H-1b]}] there exists $\gamma\in[1,d]$ such that $\mu \left(B(\theta,r) \cap S^{d-1} \right) \le C r^{\gamma -1},$
with $\gamma+\alpha>d$, and for all $s>0$, there exists $C>0$ such that:
\begin{equation}
\bar{q}(s) \le C \bar{q}(2s) \label{DOUBLING_TEMP}
\end{equation}
\end{itemize}


\item[\textbf{[H-2]}] 
Denoting by $\varphi_Z$ the L\'evy-Kintchine exponent of $(Z_t)_{t\ge 0}$,
there is $K>0$ such that :
\begin{equation}
\E\left( e^{i\langle \zeta,Z_t \rangle}\right) = e^{t\varphi_Z(\zeta)} \le e^{-Kt|\zeta|^\alpha}, \ |\zeta| >1.
\end{equation}


\item[\textbf{[H-3]}] $F: \R_+\times \R^d \rightarrow \R^d$ is Lipschitz continuous in space or measurable and bounded and
$\sigma: \R_+\times \R^d \rightarrow \R^d\otimes\R^d$ is bounded and $\eta$-H\"older continuous in space $\eta \in (0,1)$.

Also, when $\alpha \le 1$ , we assume $F=0$.

\item[\textbf{[H-4]}]  $\sigma$ is uniformly elliptic. There exists $\kappa>1$ for all $x,\xi \in \R^d$,  such that:
\begin{equation}
\kappa^{-1} |\xi|^2 \le \langle \xi, \sigma(t,x) \xi \rangle \le \kappa |\xi|^2.
\end{equation}

\item[\textbf{[H-5]}]  For all $ A \in \mathcal{B}$, Borelian, we define the measure:
\begin{equation}
\nu_t(x,A)= \nu\big( \{z \in \R^d; \ \sigma(t,x)z \in A \}\big).
\end{equation}
We assume these measures to be H\"older continuous with respect to the first parameter, that is, 
for all $A \in \mathcal{B}(\R^d)$, 
$$
|\nu_t(x,A) - \nu_t(x',A)| \le C \delta\wedge|x-x'|^{\eta(\alpha\wedge 1)}  \int_{S^{d-1}} \int_0^{+\infty}  \ind_{A}(s\theta) \frac{\bar q(s)}{s^{1+\alpha}}ds \mu(d\theta).
$$

\item[\textbf{[H-LB]}] There is a non increasing function $\underline{q}: \R_+ \rightarrow \R_+$ and $A_{low} \subset  \R^d$, such that for all $x\in A_{low}$,
\begin{eqnarray}
\nu \Big( B(x,r) \Big) &\ge& C r^\gamma\frac{\underline{q}(|x|)}{|x|^{\alpha+\gamma}}, \ \forall r >0
\end{eqnarray}

Also, for all $(t,x)\in\R_+\times \R^d$, we assume that $\sigma(t,x) A_{low} \subset A_{low}$.

\end{trivlist}

\bigskip


In the rest of this paper, we will assume that $\textbf{[H-1]}$ to $\textbf{[H-5]}$ is in force.
Also, we say that $\textbf{[H]}$ holds when $\textbf{[H-1]}$ to $\textbf{[H-5]}$ hold.
Note that assumption \textbf{[H-2]} is crucial in order to get the existence of the density.
We point out that \textbf{[H-LB]} is needed for the lower bound, and that the upper bound holds independently.

Recall that for a Markov process $(X_t)_{t\ge 0}$, the infinitesimal generator (or simply the generator) $L_t(x,\nabla_x)$ is defined as the limit:
$$
L_t(x,\nabla_x) f(x) = \lim_{s\downarrow t} \frac{\E [f(X_s) | X_t=x] - f(x)}{s-t}, \ f\in\mathcal C^{1,2}_0(\R_+\times \R^{d},\R).
$$
We use a terminology from the PDE theory to emphasise the variable on which the operator acts.
This specification will be useful later, as we will apply these operators to transition densities, with two space arguments.

Also, a probability measure $\PP$ on $\Omega = \mathcal{D}(\R_+\times \R^{d},\R)$ the space of c\`adl\`ag functions
is a solution to the martingale problem associated with $L_t(x,\nabla_x)$ if and only if
for every $x\in\R^{d}$, 
for all $f \in \mathcal C^{1,2}_0(\R_+\times \R^{d},\R)$ (twice continuously differentiable functions with compact support), denoting by $(X_t)_{t \ge 0}$ the canonical process, we have:
$$
\PP(X_t=x ) =1\ \mbox{ and } \ f(s,X_s) - \int_t^s(\partial_u+L_u(x,\nabla_x)) f(u,X_u)du \   \mbox{  is a $\PP$- martingale.}
$$

Under $\textbf{[H]}$, we are able to prove the following.
\begin{theorem}[\textbf{Weak Uniqueness}]
\label{MART_PB_THM_TEMP}
Assume \textbf{[H]} holds. The martingale problem associated with the generator $L_t(x,\nabla_x)$ of the equation \eqref{EDS1}:
$$
L_t(x,\nabla_x)\varphi(x)=\langle F(t,x), \nabla_x \varphi(x)\rangle+ \int_{\R^d} \varphi(x+\sigma(t,x)z) - \varphi(x) - \frac{\langle \sigma(t,x)z , \nabla_x \varphi(x)\rangle}{1+|z|^2} \nu(dz),
$$
admits a unique solution. 
Thus, weak uniqueness holds for \eqref{EDS1}.
\end{theorem}

Let us mention that the weak uniqueness has been derived in similar setting by various authors.
See for instance \cite{stro:75} or \cite{koma:08}.
We propose here an approach based on a parametrix technique, which derives from the works of Bass and Perkins \cite{bass:perk:09}.

Also, we have the following density estimate:

\begin{theorem}[\textbf{Density Estimates}]\label{MTHM_TEMP}
Under $\textbf{[H]}$, the unique weak solution of \eqref{EDS1} has for every $0 \le t\le s$, 
an absolutely continuous transition probability. Precisely, for all $0 \le t\le s, $ and $ x,y\in \R^{d}$,  
\begin{equation}
\PP(X_s \in dy | X_t =x)=p(t,s,x,y) dy.
\end{equation}
Assume that the function $Q$ defined below is decreasing, and fix a deterministic time horizon $T>0$.
There exists $C_1\ge 1$ depending on $T$ and the parameters in $\textbf{[H]}$, such that the following density estimates holds:

$ \forall 0\le t \le s\le T,\ \forall (x,y)\in \R^{d}, $
\begin{equation}\label{UB_DENS_SDE}
p(t,s,x,y) \le C_1\frac{ (s-t)^{-d/\alpha}}{\left(1+\frac{|y-\theta_{s,t}(x) |}{(s-t)^{1/\alpha}} \right)^{\alpha+\gamma}} Q(|y-\theta_{s,t}(x) |),
\end{equation}

where:

\begin{itemize}

\item when the drift $F$ is bounded, $\theta$ is the identity map: $\theta_{s,t}(x)=x$, and
\begin{itemize}
\item under $\textbf{[H-1a]}$, $\gamma=d$ and for all $s>0$, $Q(s) =\bar{q}(s)$, 
\item under $\textbf{[H-1b]}$, for all $s>0$, $Q(s) =\min(1,s^{\gamma-1})\bar{q}(s)$,
\end{itemize}

\item when the drift $F$ is Lipschitz continuous, $\theta_{s,t}(x)$ denotes the solution to the ordinary differential equation:
$$
\frac{d}{ds} \theta_{s,t}(x) = F(\theta_{s,t}(x)), \ \theta_{t,t}(x) =x, \ \forall 0 \le t \le s \le T,
$$
and 
\begin{itemize}
\item under $\textbf{[H-1a]}$, $\gamma=d$ and for all $s>0$, $Q(s) =\min(1,s)\bar{q}(s)$,
\item under $\textbf{[H-1a]}$, for all $s>0$, $Q(s) =\min(1,s,s^{\gamma-1})\bar{q}(s)$.
\end{itemize}

\end{itemize}

Moreover, assume \textbf{[H-LB]} holds, and that for some $C>0$,
\begin{equation}\label{CLAIM_LOWER_BOUND}
B \big( \theta_{t,T}(y) - x, C(T-t)^{1/\alpha} \big) \subset 
A_{low},
\end{equation}
there exists $C_2>1$  such that 
\begin{equation}\label{LB_DENS_SDE}
C_2^{-1}\frac{ (s-t)^{-d/\alpha}}{\left(1+\frac{|y-\theta_{s,t}(x) |}{(s-t)^{1/\alpha}} \right)^{\alpha+\gamma}}\underline{q}(|y-\theta_{s,t}(x) |) \le p(t,s,x,y).
\end{equation}
\end{theorem}

With a small abuse of language, we refer to $p(t,s,x,y)$ as the density with respect to the Lebesgue measure.
We point out that in \textbf{[H-2]}, when $\sigma \in \R$, or when $\mu$ 
is equivalent to the Lebesgue measure on $S^{d-1}$ this is actually a consequence of the H\"older continuity of $\sigma$, and the domination \textbf{[H-1]}.
However, in a very general setting, this is not a consequence of the H\"older continuity of $\sigma$, even in the case of the stable process.
Consider for instance $Z_t=(Z_t^1,Z_t^2)$, where $Z_t^i,i=1,2$ are two independent one dimensional stable processes.
In that case, $\mu = \delta_{e_1}+ \delta_{e_2}$, where $(e_1,e_2)$ forms the canonic base of $\R^2$.
Thus, the support of L\'evy measure of $Z$ is exactly the reunion $e_1 \R \cup e_2 \R$.
If we take $\sigma(t,x)$ to be a rotation. There exists $A\in \mathcal{B}(\R^d)$ and $(t,x)\in\R_+\times \R^d $ such that $\nu(A) =0$, but $\nu_t(x,A) \neq 0$, consequently \textbf{[H-5]} fails in this situation.

The condition \eqref{CLAIM_LOWER_BOUND} 
appearing for the lower bound comes from the possibly unbounded feature of the deterministic flow associated with \eqref{EDS1}.
Indeed, it states that if a neighborhood at the characteristic time scale of the flow stays in the sets of non degeneracy for $\nu$, then the lower bound holds.
Also, in \textbf{[H-LB]}, the assumption that $\sigma(t,x) \in A_{low}$ is a form of compatibility condition.
We know from Watanabe \cite{wata:07} that the spectral measure plays a key role in the density estimate.
This assumption ensures that the diffusion coefficient does not alter too much the spectral measure, so that the density estimates still holds.
Besides, as the lower bound only holds inside the set $A_{low}$, 
we obtained Aronson estimates only in those specific cases where $A_{low}=\R^d$.
Let us mention that \textbf{[H]} is satisfied in its entirety when considering a rotationally invariant stable process, or a relativistic stable process.
For the latter, the Fourier exponent writes $(|\zeta|^2+1)^{\alpha/2} -1$. In that case, we have $\gamma=d$, $A_{low}=\R^d$.
See the remark after Theorem 2 in Sztonyk \cite{szto:10}.


We point out that the upper bound \eqref{UB_DENS_SDE} is not the one established in \cite{szto:10}, rather a "degraded version" of it.
Formally, it means that some integrability is required in order to perform our techniques (precisely to correct a bad concentration index on the parametrix kernel).
Nevertheless, our approach allows to recover the existence and estimates on the density of \eqref{EDS1} when $(Z_t)_{t \ge 0}$ is a rotationally invariant stable process.
See the proof of Proposition \ref{EST_H_PROP} and Remark \ref{ROT_INV_STABLE_OK}.

Also, when $\textbf{[H-1a]}$ holds, the condition 
$$\mu \left(B(\theta,r) \cap S^{d-1} \right) \le C r^{\gamma -1},$$ actually holds with $\gamma=d$.
Besides, the function $Q$ appearing in the upper bound \eqref{UB_DENS_SDE} is decreasing
typically when considering tempering function of the form:
$$
\bar{q}(s) = \frac{1}{1+s^m},
$$
where $m$ is large enough.
Let us mention that when the spectral measure  $\mu$ has a density with respect to the Lebesgue measure on the sphere and the drift $F$ is bounded, we recover the results in  Kolokoltsov \cite{kolo:00}, without regularity assumptions.

\begin{remark}[On the constants]
We will often use the capital letter $C$ to denote a strictly positive constant that can depend on $T$ and the set of assumptions \textbf{[H]} and whose 
 value of $C$ may change from line to line. Similarly, in the temperation, we will often write $\bar{q}(|x|)$ where we actually mean $\bar{q}(C|x|)$. 
Finally, we will use the symbol $\asymp$ to denote the equivalence:
$$
f\asymp g \Leftrightarrow \exists C>1, \ C^{-1} f(x) \le g(x) \le C f(x).
$$
\end{remark}


The rest of this paper is organized as follows.
In Section \ref{parametrix_setting}, we set up formally the parametrix technique, 
and give the estimates permitting the convergence of the parametrix series.
Section \ref{Proof of the estimates} is a technical section and is divided in five subsections.
First, in Subsection \ref{EST_FRZN_DENS_SECTION} we prove estimates on the frozen density.
In Subsection \ref{SECT_SMTH_H}, we investigate the parametrix kernel and its smoothing properties.
In Subsection \ref{PROOF_MART_PB}, we tackle the well-posedness of the martingale problem, using estimates provided by the two previous subsections.
Next, in Subsection \ref{PROOF_LEMME_IT_KER}, we prove the estimates giving the convergence of the parametrix series.
Finally, in Subsection \ref{proof_lower_bound}, we investigate the lower bound \eqref{LB_DENS_SDE}.

\section{The Parametrix Setting}
\label{parametrix_setting}

We present here a continuity technique known as the Parametrix.
Our approach is close to the one of Mc Kean and Singer \cite{mcke:sing:67}.
The strategy is to approximate the solution of \eqref{EDS1} by the solution of a simpler equation and control the distance in some sense between the two processes.
First of all, let us define the \textit{proxy} we will use.

Let $y\in \R^d$ be  an arbitrary point. 
When $F$ is only measurable and bounded, we take as frozen process $\tilde{X}_t = x+\int_0^t\sigma(u,y)Z_u$.

Assume $F$ to be Lipschitz continuous. 
Let $\theta_{t,s}$ be the flow associated with the deterministic differential equation:
$$
\frac{d}{dt} \theta_{t,s}(x) = F(t,\theta_{t,s}(x)), \ \theta_{s,s}(x)=x, \ 0\le t,s \le T.
$$
We will often refer to $\theta_{t,s}(y)$ as the transport of $y$ by the deterministic part of \eqref{EDS1}.
Recall $T$ is the deterministic time horizon. Fix $y\in \R^d$, a terminal point and $t \in [0,T]$, and $x \in \R^d$
initial time and position, we define the \textit{frozen process} $(\tilde{X}^{T,y}_s)_{s\in [t,T]}$ as:
\begin{eqnarray}\label{EDS_GEL}
\tilde{X}_s^{T,y} &=& x + \int_t^s F(u,\theta_{u,T}(y)) du + \int_t^s \sigma(u,\theta_{u,T}(y)) dZ_u.
\end{eqnarray}

We point out that the transport of the terminal point in the drift part comes for the unbounded character of the drift coefficient.
Also, in diffusion coefficient $\sigma$, the presence of the transport ensures the compatibility between the estimates on the frozen process and the parametrix kernel (see Propositions \ref{EST_FRZ_DENS_PROP} and \ref{EST_H_PROP}).
This technique of freezing along the curves of the deterministic system has first been introduced in Konakov Menozzi and Molchanov \cite{kona:meno:molc:10}, and has been formalised in Delarue and Menozzi \cite{dela:meno:10}.
In those works, the presence of the flow has another purpose, namely to account for the degenerate feature of their equations.
However, as a by-product, it allows to handle the unbounded feature of the drift in a non degenerate setting.

 We could restrict ourselves to functions $F$ that are H\"older continuous. In this case, existence of the flow $\theta$ is given by the Cauchy Peano theorem. However, the lack of uniqueness poses the problem of the definition of $\theta_{t,s}$, so we decided to assume Lipschitz continuity instead.
Anyhow, in the case where $F$ is H\"older continuous, we expect some kind of regularization by the noise, as we recover weak uniqueness, see e.g. Bafico and Baldi \cite{bafi:bald:82}, or Delarue and Flandoli
\cite{dela:flan:14} for recent developments.

It is clear from the definition of $(\tilde{X}^{T,y}_s)_{s\in [t,T]}$ and assumptions \textbf{[H-1]} (domination of the L\'evy measure) and \textbf{[H-2]} (non degeneracy of the spectral measure) that $\tilde X^{T,y}$ has a density with respect to the Lebesgue measure.
We denote the latter: 
$$
\PP (\tilde{X}^{T,y}_s \in dz | \tilde{X}_t^{T,y} = x)= \tilde{p}^{T,y}(t,s,x,z)dz, \ s\in(t,T].
$$

To get an explicit representation for it, we proceed by a Fourier inversion.
Observe first that the Fourier transform of $\tilde{X}^{T,y}_s$ actually writes:
$$
\E(e^{i \langle p, \tilde{X}^{T,y}_s\rangle }) = e^{i \langle p, x + \int_t^s F(u,\theta_{u,T}(y))du \rangle}
\exp \left( \int_t^s \varphi_Z(\sigma(u,\theta_{u,T}(y))^*p )\right),
$$
where we denoted by $\sigma(u,\theta_{u,T}(y))^*$ the transpose of $\sigma(u,\theta_{u,T}(y))$, and 
$\varphi_Z$ is the L\'evy-Kintchine exponent of $Z$.
Due to assumptions \textbf{[H-1]} and \textbf{[H-2]}, this exponent is integrable and the frozen density actually writes:

\begin{eqnarray}\label{DENS_GEL_REP}
\tilde{p}^{T,y}(t,s,x,z) &=& \frac{1}{(2\pi)^d} \int_{\R^d}dp e^{-i \langle p, z- x - \int_t^s F(u,\theta_{u,T}(y))du \rangle} \nonumber\\
&&\times \exp \left( \int_t^s du\int_{\R^d} e^{ i\langle p, \sigma(u,\theta_{u,T}(y)) \xi \rangle } -1 - i\frac{ \langle p, \sigma(u,\theta_{u,T}(y)) \xi \rangle}{1+|\xi|^2} \nu(d\xi) \right).
\end{eqnarray}

We will often denote $\tilde{p}(t,T,x,y) =\tilde{p}^{T,y}(t,T,x,y)$, namely, we omit the superscript $T,y$ when the freezing parameters and the points where the density is considered are the same.
Observe that in this case, we have
\begin{equation}\label{FLOW_REVERSE_OK}
y-\int_t^T F(u,\theta_{u,T}(y)) du-x= \theta_{t,T}(y)-x.
\end{equation}


\begin{proposition}
Assume that the semigroup $(P_{s,t})_{0 \le s,t \le T}$ generated by the solution $(X_s)_{s \in [t,T]}$ to \eqref{EDS1} is Feller.
We have the following \textit{formal} representation. For all $t > 0,\ (x,y)\in (\R^{d})^2 $ and any bounded measurable $f:\R^{d}\rightarrow \R$:
\begin{equation}
\label{SG_PARAM}
P_{T,t}f(x)=\E[f(X_T)|X_t=x]=\int_{\R^{d}} \left(\sum_{r=0}^{+\infty} (\tilde{p} \otimes H^{(r)}) (t,T,x,y) \right)f(y)dy,
\end{equation}
where $H$ is the parametrix kernel:
\begin{equation}
\forall 0 \le t \le s \le T, \ (x,y)\in (\R^{d})^2,\ H(t,T,x,y) := (L_t(x,\nabla_x) - L_t(\theta_{t,T}(y),\nabla_x) )\tilde{p}^{T,y}(t,T,x,y).
\end{equation}
The notation $\otimes  $ stands for the time space convolution: 
$$
f \otimes g(t,T,x,y) = \int_t^T du\int_{\R^{d}} dz f(t,u,x,z)g(u,T,z,y).
$$
Besides, $\tilde{p} \otimes H^{(0)} = \tilde{p}$ and $\forall r \in \N, \ H^{(r)} (t,T,x,y) =H^{(r-1)} \otimes H(t,T,x,y)$.

Furthermore, when the above representation can be justified, it yields the existence as well as a representation for the density of the initial process. Namely $\PP[X_T\in dy|X_t=x]=p(t,T,x,y)dy $ where :
\begin{equation} \label{SERIE_PARAM_TEMP}
\forall t> 0,\ (x,y)\in (\R^{d})^2,\ p(t,T,x,y) = \sum_{r=0}^{+\infty} (\tilde{p} \otimes H^{(r)}) (t,T,x,y).
\end{equation}
\end{proposition}

\begin{proof}
Let us first emphasize that the density $\tilde{p}^{T,y}(t,s,x,z)$ of $\tilde X_{s}^{t,x,T,y} $ at point $z$ solves the Kolmogorov \textit{backward} equation:
\begin{equation}\label{CHAP_KOL_BWD}
\frac{\partial \tilde{p}^{T,y}}{\partial t}(t,s,x,z) = -L_t(\theta_{t,T}(y), \nabla_x) \tilde{p}^{T,y}(t,s,x,z), \mbox{ for all $t<s$, $(x,z) \in \R^{nd}\times \R^{nd}$, $\lim_{t\uparrow s}\tilde{p}^{T,y}(t,s,\cdot,z) = \delta_{z}(\cdot)$}.
\end{equation}
Here, the notations $\nabla_x$ highlights the fact that $L_t(\theta_{t,T}(y), \nabla_x)$ acts on the variable $x$. Let us now introduce the family of operators $(\tilde P_{t,s})_{0\le t\le s} $. For $0\le t\le s $ and any bounded measurable function $f:\R^{nd}\rightarrow \R$:
\begin{eqnarray}
\label{SG_TILDE}
\tilde P_{t,s}f(x):=\int_{\R^{nd}}\tilde p(t,T,x,y)f(y)dy:=\int_{\R^{nd}}\tilde p^{T,y}(t,T,x,y)f(y)dy.
\end{eqnarray}
Observe that the family $(\tilde P_{t,s})_{0\le t\le s}  $ is not a two-parameter semigroup, because of the integration with respect to the freezing parameter $y$. Anyhow, we can still establish, see Lemma \ref{convergence_dirac_temp}, that for a continuous $f$:
\begin{equation}
\label{LIM_SG_TILDE}
\lim_{s\rightarrow t}\tilde P_{s,t}f(x)=f(x).
\end{equation}
This convergence is not a direct consequence of the bounded convergence theorem since the freezing parameter is also the integration variable.
It rather follows from the estimates on the Frozen density $\tilde p^{T,y}(t,T,x,y)$.

The boundary condition \eqref{LIM_SG_TILDE} and the Feller property yield:
$$
(P_{t,T}-\tilde P_{t,T})f(x) = \int_t^T du \frac{\partial}{\partial u}\biggl\{P_{t,u}(\tilde P_{u,T}f(x))\biggr\}.
$$
Computing the derivative under the integral leads to:
$$
(P_{t,T}-\tilde P_{t,T})f(x) = \int_t^T du \biggl\{\partial_u P_{t,u}(\tilde P_{u,T}f(x))+P_{t,u}(\partial_u(\tilde P_{u,T}f(x)))\biggr\}.
$$
Using the Kolmogorov equation \eqref{CHAP_KOL_BWD} and the Chapman-Kolmogorov relation $\partial_{u}P_{t,u}\varphi(x) =P_{t,u}(L_u(x,\nabla_x) \varphi(x)),\ \forall \varphi\in C_b^{2}(\R^{nd},\R) $  we get:
\begin{eqnarray*}
(P_{t,T}-\tilde P_{t,T})f(x) = \int_t^T du  P_{t,u}\Big(L_u(x,\nabla_x) \tilde{P}_{u,T}f\Big)(x) - P_{t,u}\left(\int_{\R^{nd}}f(y)  L_u(\theta_{u,T}(y),\nabla_x)\tilde p(u,T,\cdot,y) dy \right)(x).
\end{eqnarray*}
Define now the operator:
\begin{equation}
\label{H_CUR}
\mathcal{H}_{u,T}(\varphi)(z) := \int_{\R^{nd}}\varphi(y) \Big(L_u(x,\nabla_x)-L_u(\theta_{u,T}(y),\nabla_x) \Big)\tilde p(u,T,z,y)dy=\int_{\R^{nd}}\varphi(y) H(u,T,z,y)dy.
\end{equation}
We can thus rewrite:
$$
P_{t,T}f(x) = \tilde P_{t,T}f(x) + \int_t^T P_{t,u}\big(\mathcal{H}_{u,T}(f) \big)(x)du.
$$
The idea is now to reproduce this procedure for $P_{t,u} $ applied to $\mathcal{H}_{u,T}(f) $.
This recursively yields the formal representation:  
$$
P_{t,T}f(x) = \tilde P_{t,T}f(x)+ \sum_{r \ge 1} \int_t^T du_1 \int_t^{u_1} du_2 \dots \int_t^{u_{r-1}} du_r \tilde{P}_{t,u_r}
\big(\mathcal{H}_{u_r,u_{r-1}}\circ \dots \circ \mathcal{H}_{u_1,T}\big)(f)(x).
$$
Equation \eqref{SG_PARAM} then formally follows from the following identification. For all $r\in \N^*$:
 $$
 \int_t^T du_1 \int_t^{u_1} du_2 \dots \int_t^{u_{r-1}} du_r \tilde{P}_{t,u_r}
\big(\mathcal{H}_{u_r,u_{r-1}}\circ \dots \circ \mathcal{H}_{u_1,T}\big)(f)(x)du = \int_{\R^{nd}}f(y) \tilde p \otimes H^{(r)}(t,T,x,y)dy.
 $$
 
We can proceed by immediate induction:
 \begin{eqnarray}
 &&\int_t^T du_1 \int_t^{u_1} du_2 \dots \int_t^{u_{r-1}} du_r \tilde{P}_{t,u_r}
\big(\mathcal{H}_{u_r,u_{r-1}}\circ \dots \circ \mathcal{H}_{u_1,T}\big)(f)(x)du \nonumber\\
&=& 
\int_t^T du_1 \int_t^{u_1} du_2 \dots \int_t^{u_{r-1}} du_r \int_{\R^{nd}}dz\mathcal{H}_{u_r,u_{r-1}}\circ \dots \circ \mathcal{H}_{u_1,T}(f)(z) \tilde p(t,u_{r},x,z) \label{POUR_BOUCLER}\\
&\overset{\eqref{H_CUR}}{=}&\int_t^T du_1 \int_t^{u_1} du_2 \dots \int_t^{u_{r-1}} du_r \int_{\R^{nd}}dz\int_{\R^{nd}}dy \mathcal{H}_{u_{r-1},u_{r-2}}\circ \dots \circ \mathcal{H}_{u_1,T}(f)(y) H(u_r,u_{r-1},z,y)\tilde p(t,u_{r},x,z) \nonumber \\
&=&\int_t^T du_1 \int_t^{u_1} du_2 \dots \int_t^{u_{r-2}} du_{r-1}\int_{\R^{nd}}dy \mathcal{H}_{u_{r-1},u_{r-2}}\circ \dots \circ \mathcal{H}_{u_1,T}(f)(y) \tilde p_\alpha \otimes H (t,u_{r-1},x,y) .\nonumber
 \end{eqnarray}
 
Thus, we can iterate the procedure from \eqref{POUR_BOUCLER} with $\tilde p \otimes H$ instead of $\tilde p$.

\end{proof}

The existence of the density for the solution of \eqref{EDS1} will follow from the convergence of the parametrix series.
In the following, we will denote
\begin{equation}
\bar{p}(t,T,x,y) = \frac{(T-t)^{-d/\alpha}}{\left( 1+ \frac{|y-\theta_{T,t}(x)|}{(T-t)^{1/\alpha}}\right)^{\alpha+\gamma}} \bar{q}(|y-\theta_{T,t}(x)|).
\end{equation}
This is the upper bound on the Frozen density under \textbf{[H]} derived by Sztonyk \cite{szto:10}, adapted to our possible unbounded drift case.
We prove that this upper bound holds for the frozen density in Section \ref{Proof of the estimates}.

The following lemma proves the convergence of the series \eqref{SERIE_PARAM_TEMP}.

\begin{lemma}[\textbf{Control of the iterated kernels}]\label{LEMME_IT_KER_TEMP}
There exist $C_{\ref{LEMME_IT_KER_TEMP}}=C_{\ref{LEMME_IT_KER_TEMP}}(T) >0$, $\omega\in (0,1]$ s.t.
for all $t\in [0,T]$, $(x,y)\in (\R^{d})^2 $: 
\begin{eqnarray}
|\tilde p\otimes H(t,T,x,y)|&\le &\frac{C_{\ref{LEMME_IT_KER_TEMP}}}{\omega}\Big ( (T-t)^\omega \bar {p} (t,T,x,y)+ \rho(t,T,x,y) \Big), \label{EST_LEMME_IT_KER_1}\\
| \rho \otimes H(t,T,x,y)|&\le &\frac{C_{\ref{LEMME_IT_KER_TEMP}}}{\omega}  (T-t)^\omega  \bar {p}(t,T,x,y), \label{EST_LEMME_IT_KER_2}
\end{eqnarray}
where we denoted
$\rho(t,T,x,y) = \delta \wedge |x-\theta_{t,T}(y)|^{\eta(\alpha \wedge 1)} \bar {p}(t,T,x,y)$.
Now for all $k\ge 1$,
\begin{eqnarray}
|\tilde p\otimes H^{(2k)}(t,T,x,y)|\le (C_{\ref{LEMME_IT_KER_TEMP}})^{2k} \frac{(T-t)^{k\omega}}{k! \omega^{2k}} \Big( (T-t)^{k\omega} \bar {p}(t,T,x,y) + (\bar {p} + \rho)(t,T,x,y) \Big), \label{EST_LEMME_IT_KER_3} \\
|\tilde p\otimes H^{(2k+1)}(t,T,x,y)|\le (C_{\ref{LEMME_IT_KER_TEMP}})^{2k+1} \frac{(T-t)^{k\omega}}{(k+1)! \omega^{2k+1}} \Big( (T-t)^{(k+1)\omega} \bar {p}+ (T-t)^{\omega'}(\bar {p} +\rho)+ \rho \Big)(t,T,x,y), \label{EST_LEMME_IT_KER_4}
\end{eqnarray}
where we set $\omega' = \omega\ind_{T-t\le 1} + k\omega \ind_{T-t >1}>0$. 
\end{lemma}

The above controls allow to derive under the sole assumption $\H$ the convergence of the parametrix series (thus, existence of the density for the solution of \eqref{EDS1}), and the upper bound \eqref{UB_DENS_SDE} for the sum of the parametrix series \eqref{SERIE_PARAM_TEMP}. 

\begin{proof}
Let us denote by
$$\bar{H}(t,T,x,y) = \frac{\delta\wedge |x-\theta_{t,T}(y)|^{\eta(\alpha \wedge 1)} }{T-t} \bar{p}(t,T,x,y).$$

We prove in Section \ref{SECT_SMTH_H} that $\bar{H}$ is actually the upper bound for the kernel $H$.
The proof of Lemma \ref{LEMME_IT_KER_TEMP} relies on the two following estimates:
\begin{eqnarray}
&&\int_{\R^d} \bar{p}(t,u,x,z) \bar{H}(u,T,z,y) dz 
\le C_0 \Big( (T-u)^{\omega-1} + (u-t)^{\omega-1}+ \frac{\delta\wedge |x-\theta_{t,T}(y)|^{\eta(\alpha \wedge 1)} }{T-t}\Big) \bar{p}(t,T,x,y) \label{EST_PREAL_1}\nonumber\\
\\
&&\int_{\R^d} \rho(t,u,x,z)  \bar{H}(u,T,z,y) dz
\le C_0\Big( (T-u)^{\omega-1} + (u-t)^{\omega-1}\Big) \bar{p}(t,T,x,y) \label{EST_PREAL_2}.
\end{eqnarray}

Integrating these estimates in time yields \eqref{EST_LEMME_IT_KER_1} and  \eqref{EST_LEMME_IT_KER_2}, setting $C_{\ref{LEMME_IT_KER_TEMP}} = 2C_0$.
We postpone the proof of these important estimates in Section \ref{Proof of the estimates}, as it is technical and relies on sharp estimates on the frozen density and on the parametrix kernel (see Lemmas \ref{FIRST_STEP} and \ref{SECOND_STEP}).
Assuming estimates \eqref{EST_PREAL_1} and  \eqref{EST_PREAL_2}, 
we prove estimates  \eqref{EST_LEMME_IT_KER_3} and  \eqref{EST_LEMME_IT_KER_4} by induction.

\textbf{Initialization:}\\
Since $(T-t)^{\omega'}(\bar p + \rho)\geq 0$, we clearly have:
$$
|\tilde p\otimes H(t,T,x,y) | \leq C_{\ref{LEMME_IT_KER_TEMP}} \Big((T-t)^\omega \bar {p}+\rho + (T-t)^\omega(\bar {p}+ \rho) \Big)(t,T,x,y).
$$
We now turn to the estimate for $|\tilde{p} \otimes H^{(2)}|$.
Starting from \eqref{EST_LEMME_IT_KER_1}, we use  equations \eqref{EST_PREAL_1} and \eqref{EST_PREAL_2}, we have to write:

\begin{eqnarray*}
|\tilde{p} \otimes H^{(2)}(t,T,x,y)| &\leq& 
\frac{C_{\ref{LEMME_IT_KER_TEMP}}}{\omega} \int_t^T du (u-t)^{\omega} \int_{\R^d} \bar{p}(t,u,x,z) \frac{\delta\wedge |x-\theta_{u,T}(y)|^{\eta(\alpha \wedge 1)}}{T-u} \bar{p}(u,T,z,y) dz\\
&&+\frac{C_{\ref{LEMME_IT_KER_TEMP}}}{\omega}  \int_t^T du \int_{\R^d}\rho(t,u,x,z)\frac{\delta\wedge |x-\theta_{u,T}(y)|^{\eta(\alpha \wedge 1)}}{T-u} \bar{p}(u,T,z,y) dz\\
&\le& \frac{C_{\ref{LEMME_IT_KER_TEMP}}}{\omega}  C_0 \int_t^T du (u-t)^{\omega} \Big( (T-u)^{\omega-1} + (u-t)^{\omega-1}+ \frac{\delta\wedge |x-\theta_{t,T}(y)|^{\eta(\alpha \wedge 1)} }{T-t}\Big) \bar{p}(t,T,x,y)\\
&&+\frac{C_{\ref{LEMME_IT_KER_TEMP}}}{\omega}   C_0 \left(\int_t^T du (T-u)^{\omega-1} + (u-t)^{\omega-1} \right) \bar{p}(t,T,x,y)\\ 
&\le& \frac{C_{\ref{LEMME_IT_KER_TEMP}}}{\omega}C_0 \biggl[ \Big(B(\omega+1,\omega) + \frac{1}{2\omega} \Big) (T-t)^{2\omega} \bar{p}(t,T,x,y) + \frac{(T-t)^\omega}{\omega+1} \rho(t,T,x,y) + 2\frac{(T-t)^\omega}{\omega} \bar{p}(t,T,x,y) \biggr]\\
&\le& \frac{(C_{\ref{LEMME_IT_KER_TEMP}})^2}{\omega^2}(T-t)^\omega \Big( (T-t)^\omega\bar{p} + (\bar{p}+\rho)  \Big)(t,T,x,y).
\end{eqnarray*}

To get to the last estimate, observe that the biggest contribution amongst the constant multipliers is $\frac{2}{\omega}$.
We also recall that  $C_{\ref{LEMME_IT_KER_TEMP}}= 2C_0$.


\textbf{Induction}:\\
Suppose that the estimate for $2k$ holds. Let us prove the estimate for $2k+1$.
Recalling that $\bar{p}$ and $\bar{H}$ are the upper bounds for $\tilde{p}$ and  $H$ respectively, we write:
\begin{eqnarray*}
&&|\tilde{p}\otimes H^{(2k+1)}(t,T,x,y)|\\
&\le& \frac{(C_{\ref{LEMME_IT_KER_TEMP}})^{2k}}{k! \omega^{2k}} \left( \int_t^T du (u-t)^{2k\omega+1} \int_{\R^d} \bar{p}(t,u,x,z) \bar{H}(u,T,z,y) dz
+\int_t^T du (u-t)^{k\omega} \int_{\R^d} (\bar{p}+ \rho)(t,u,x,z)\bar{H}(u,T,z,y) dz \right)\\
&\le&  \frac{(C_{\ref{LEMME_IT_KER_TEMP}})^{2k}}{k! \omega^{2k+1} }  
\int_t^T du (u-t)^{2k\omega} C_0 \Big( (T-u)^{\omega-1} + (u-t)^{\omega-1}+ \frac{\delta\wedge |x-\theta_{t,T}(y)|^{\eta(\alpha \wedge 1)} }{T-t}\Big) \bar{p}(t,T,x,y) \\
&&+\frac{(C_{\ref{LEMME_IT_KER_TEMP}})^{2k}}{k! \omega^{2k+1} }  \int_t^T du (u-t)^{k\omega}C_0 \Big(  2((T-u)^{\omega-1} + (u-t)^{\omega-1})+ \frac{\delta\wedge |x-\theta_{t,T}(y)|^{\eta(\alpha \wedge 1)} }{T-t}\Big) \bar{p}(t,T,x,y)\\
&\le& \frac{(C_{\ref{LEMME_IT_KER_TEMP}})^{2k}}{k! \omega^{2k} } C_0\left[ \Big( B(2k\omega+1,\omega) +\frac{1}{(2k+1)\omega} \Big) (T-t)^{(2k+1)\omega} \bar{p}(t,T,x,y) + \frac{(T-t)^{2k\omega}}{2k\omega+1}\rho(t,T,x,y)\right]\\
&&+\frac{(C_{\ref{LEMME_IT_KER_TEMP}})^{2k}}{k! \omega^{2k} } C_0 \left[ 2\Big( B(k\omega+1,\omega) +\frac{1}{(k+1)\omega} \Big) (T-t)^{(k+1)\omega} \bar{p}(t,T,x,y) + \frac{(T-t)^{k\omega}}{k\omega+1}\rho(t,T,x,y)\right]\\
&\le&  \frac{(C_{\ref{LEMME_IT_KER_TEMP}})^{2k+1}}{(k+1)! \omega^{2k+1} } (T-t)^{k\omega} \Big( (T-t)^{(k+1)\omega} \bar{p} + (T-t)^{\omega'}(\bar{p} + \rho) + \rho\Big)(t,T,x,y).
\end{eqnarray*}

To get the last inequality, observe that from all the constant multipliers, the biggest one is $\frac{1}{(k+1)\omega}$.

Assume now that the estimate for $2k+1$ holds. We prove the estimate for $2k+2$.
\begin{eqnarray*}
&&|\tilde{p}\otimes H^{(2k+2)}(t,T,x,y)|\\
&\le& \frac{(C_{\ref{LEMME_IT_KER_TEMP}})^{2k+1}}{k! \omega^{2k+1}} \biggl( \int_t^T du (u-t)^{(2k+1)\omega} \int_{\R^d} \bar{p}(t,u,x,z) \bar{H}(u,T,z,y) dz\\
&&\quad+\int_t^T du (u-t)^{k\omega+ \omega'} \int_{\R^d} (\bar{p}+ \rho)(t,u,x,z)\bar{H}(u,T,z,y) dz \\
&&\quad+\int_t^T du (u-t)^{k\omega} \int_{\R^d} \rho(t,u,x,z)\bar{H}(u,T,z,y) dz\biggr)\\
&\le& \frac{(C_{\ref{LEMME_IT_KER_TEMP}})^{2k+1}}{k! \omega^{2k+1}} \int_t^T du (u-t)^{(2k+1)\omega} C_0 \Big( (T-u)^{\omega-1} + (u-t)^{\omega-1}+ \frac{\delta\wedge |x-\theta_{t,T}(y)|^{\eta(\alpha \wedge 1)} }{T-t}\Big) \bar{p}(t,T,x,y)
\\
&&+\frac{(C_{\ref{LEMME_IT_KER_TEMP}})^{2k+1}}{k! \omega^{2k+1} }  \int_t^T du (u-t)^{k\omega+\omega'}C_0 \Big(  2((T-u)^{\omega-1} + (u-t)^{\omega-1})+ \frac{\delta\wedge |x-\theta_{t,T}(y)|^{\eta(\alpha \wedge 1)} }{T-t}\Big) \bar{p}(t,T,x,y)\\
&&+\frac{(C_{\ref{LEMME_IT_KER_TEMP}})^{2k+1}}{k! \omega^{2k+1} }  \int_t^T du (u-t)^{k\omega}C_0 \Big(  (T-u)^{\omega-1} + (u-t)^{\omega-1}\Big) \bar{p}(t,T,x,y)\\
&\le& \frac{(C_{\ref{LEMME_IT_KER_TEMP}})^{2k+1}}{k! \omega^{2k+1}}C_0 \left(
 \Big( B((2k+1)\omega+1,\omega) + \frac{1}{(2k+2)\omega}\Big)(T-t)^{(2k+2)\omega} \bar{p}(t,T,x,y) +\frac{(T-t)^{(2k+1)\omega}}{(2k+1)\omega+1} \rho(t,T,x,y) \right)\\
&&+\frac{(C_{\ref{LEMME_IT_KER_TEMP}})^{2k+1}}{k! \omega^{2k+1} } C_0
\left( 
\Big( 2B(k\omega+ \omega'+1,\omega) + \frac{2}{(k+1)\omega+\omega'} \Big) (T-t)^{(k+1)\omega+\omega'}\bar{p}(t,T,x,y)
\frac{(T-t)^{k\omega+\omega'}}{k\omega+\omega'+1} \rho(t,T,x,y)
\right)\\
&&+\frac{(C_{\ref{LEMME_IT_KER_TEMP}})^{2k+1}}{k! \omega^{2k+1} }C_0 
\left(  B(k\omega+1,\omega) + \frac{1}{(k+1)\omega}\right)
 (T-t)^{(k+1)\omega}\bar{p}(t,T,x,y)\\
 &\le&  \frac{(C_{\ref{LEMME_IT_KER_TEMP}})^{2k+2}}{(k+1)! \omega^{2k+2} }(T-t)^{(k+1)\omega} \left( (T-t)^{(k+1)\omega} \bar{p}(t,T,x,y) + (\bar{p}+\rho)(t,T,x,y) \right).
\end{eqnarray*}

To obtain the last inequality, and the estimate for $\tilde{p}\otimes H^{2(k+1)}$, we factorised by  $(T-t)^{(k+1)\omega}$.
Also, the biggest constant multiplier is $\frac{1}{(k+1)\omega}$.

\end{proof}

\section{Proof of the estimates.}
\label{Proof of the estimates}

In order for the parametrix technique to be successful, we must obtain some sharp estimates on the quantities involved in the parametrix expansion \eqref{SERIE_PARAM_TEMP}.
This is usually done in two parts, first, we give two sided estimates on the density of the frozen process, as well as a similar upper bound on the parametrix kernel $H$, up to a time singularity.
Then, we prove that those bounds yield a smoothing effect in time for the time space convolution
$\tilde p\otimes H $ appearing in \eqref{SERIE_PARAM_TEMP}.


\subsection{Estimates on the Frozen Density}
\label{EST_FRZN_DENS_SECTION}
We first give the estimates on the frozen density.

\begin{proposition}\label{EST_FRZ_DENS_PROP}
Assume $\H$ is in force. There exists $C>1$ s.t. for all $t\in[0,T]$, $(x,y)\in (\R^{d})^2 $:
\begin{equation}\label{UB_FROZEN_DENS}
\tilde{p}^{T,y}(t,s,x,z) \le C \frac{(s-t)^{-d/\alpha}}{ \left(1+\frac{|z-x-\int_t^s F(u,\theta_{u,T}(y))du |}{(s-t)^{1/\alpha}} \right)^{\alpha+\gamma}} \bar{q}\left(C^{-1}\left|z-x-\int_t^s F(u,\theta_{u,T}(y))du\right|\right).
\end{equation}
Moreover, assume $\textbf{[H-LB]}$ holds.
Then, for all $z-x-\int_t^s F(u,\theta_{u,T}(y))du\in A_{low}$, the lower bound holds:
\begin{equation}
\frac{C^{-1}(s-t)^{-d/\alpha}}{ \left(1+\frac{|z-x-\int_t^s F(u,\theta_{u,T}(y))du |}{(s-t)^{1/\alpha}} \right)^{\alpha+\gamma}} \underline{q}\left(C\left|z-x-\int_t^s F(u,\theta_{u,T}(y))du \right|\right) \le
\tilde{p}^{T,y}(t,s,x,z).
\end{equation}

\end{proposition}

\begin{proof}
We prove these estimates in the lines of Sztonyk \cite{szto:10}.
The idea consist in splitting large jumps and small jumps at the characteristic time scale and exploit the L\'evy-It\^o decomposition.
Fix $t,T\in \R_+$ and $y\in \R^d$.
Observe that since the drift part in the frozen process is deterministic, it suffices to prove the estimates or:
\begin{eqnarray*}
\Lambda_s= \int_t^s \sigma(u,\theta_{u,T}(y)) dZ_u. 
\end{eqnarray*}
We point out that $\Lambda_s = \Lambda_s(t,T,y)$, where $t,T,y$ are fixed.
The Fourier transform of $\Lambda_s$ writes:
$$
\E(e^{i\langle \zeta , \Lambda_s\rangle})=
\exp \left( \int_t^s du \int_{\R^d} e^{ i\langle \zeta, \sigma(u,\theta_{u,T}(y)) \xi \rangle } -1 - i \langle \zeta, \sigma(u,\theta_{u,T}(y)) \xi \rangle  \ind_{\{|\xi| \le t^{1/\alpha}\}} \nu(d\xi) \right).
$$
Changing variables in the time integral to $v\in [0,1]$ and setting $\sigma_v =\sigma((s-t)v+t,\theta_{(s-t)v+t,T}(y))$, we obtain:
$$
\E(e^{i\langle \zeta , \Lambda_s\rangle})=
\exp \left( (s-t)\int_0^1 du \int_{\R^d} e^{ i\langle \zeta, \sigma_v \xi \rangle } -1 - i \langle \zeta, \sigma_v \xi \rangle  \ind_{\{|\xi| \le t^{1/\alpha}\}} \nu(d\xi) \right).
$$
Now, defining $\nu_\mathcal{S}$ to be the image measure of $dv \nu(d\xi)$ by the application 
$(v,\xi) \mapsto \sigma_v \xi$, we obtain:
$$
\E(e^{i\langle \zeta , \Lambda_s\rangle})=
\exp \left( (s-t) \int_{\R^d} e^{ i\langle \zeta, \eta \rangle } -1 - i \langle \zeta,\eta \rangle  \ind_{\{|\eta| \le t^{1/\alpha}\}} \nu_\mathcal{S}(d\eta) \right),
$$
which is the Fourier transform of some L\'evy process $(\mathcal{S}_u)_{u\ge 0}$, with L\'evy measure $\nu_{\mathcal{S}}$, at time $s-t$.
In other words, the marginals of $(\Lambda_s)_{s \in [t,T]}$ corresponds to the marginals of $(\mathcal{S}_{s-t})_{s \in [t,T]}$:
\begin{equation}\label{ID_LOI_MARG}
 \forall s \in [t,T], \ \Lambda_s \overset{ {\rm (law)}}{=} \mathcal{S}_{s-t}.
\end{equation}

The idea is now to work with the process $(\mathcal{S}_u)_{u\ge 0}$, and prove that it satisfies the assumptions \textbf{[H]}.
Specifically, we prove that \textbf{[H-1]} and \textbf{[H-2]} holds for $\nu_\mathcal{S}$, and that when 
\textbf{[H-LB]} holds for $\nu$, it holds as well for $\nu_\mathcal{S}$.


Let $A \in \mathcal{B}(\R^d)$.
By definition of $\nu_{\mathcal{S}}$, we have:
$$
 \nu_{\mathcal S}(A) = \int_0^1 \int_{\R^{d}} \ind_{\{ \sigma_v \xi \in A\}} \nu(d\xi) dv.
$$
From the tempered stable domination, we deduce:
$$
 \nu_{\mathcal S}(A) \le \int_0^1\int_{S^{d-1}} \int_0^{+\infty}  \ind_{\{ \sigma_v s \varsigma \in A\}} 
 \frac{\bar{q}(s)ds}{s^{1+\alpha}} \mu(d\varsigma) dv.
$$
For fixed $v,\varsigma$, we change the variables in the integral in $ds$ to $\rho = s | \sigma_v \varsigma|$.
Observe that from the uniform ellipticity of $\sigma$ and the doubling property of $\bar{q}$, we have $\bar{q}\left(\frac{\rho}{|\sigma_v \varsigma|}\right)\le C\bar{q}(\rho)$.
It yields:
\begin{eqnarray*}
 \nu_{\mathcal S}(A) 
 \le  \int_0^1\int_{S^{d-1}} \int_0^{+\infty}  \ind_{\{ \rho \frac{ \sigma_v \varsigma}{|\sigma_v \varsigma|} \in A\}} 
 \frac{\bar{q}(\rho)d\rho}{\rho^{1+\alpha}} |\sigma_v \varsigma|^\alpha\mu(d\varsigma) dv.
\end{eqnarray*}

We now define $ \mu_{\mathcal{S}}(d\varsigma)$ to be the image measure of $|\sigma_v \varsigma|^\alpha \mu(d\varsigma) dv$ (measure on $[0,1]\times S^{d-1}$) by the application 
$(v,\varsigma) \mapsto  \frac{ \sigma_v \varsigma}{|\sigma_v \varsigma|}$.
We thus obtain:
\begin{eqnarray*}
 \nu_{\mathcal S}(A)  \le  \int_{S^{d-1}} \int_0^{+\infty}  \ind_{\{ \rho \zeta \in A\}} 
 \frac{\bar{q}(\rho)d\rho}{\rho^{1+\alpha}} \mu_{\mathcal S}(d \zeta).
\end{eqnarray*}
Consequently, \textbf{[H-1]} holds for $\nu_{\mathcal{S}}$.
Observe that by construction, we have that \textbf{[H-2]} holds for $\mu_\mathcal{S}$.
Therefore, from Stzonyk \cite{szto:10}, denoting $p_{\mathcal S}(u,\cdot)$ the density of 
$\mathcal{S}_u$, the following upper bound holds:
$$
p_{\mathcal S}(u,z) \le C\frac{u^{-d/\alpha} }{\left(1+ \frac{|z|}{u^{1/\alpha}}\right)^{\alpha +\gamma}} \bar{q}(|z|).
$$
We deduce the estimate for $\Lambda_s$, 
and the upper bound on $\tilde{p}^{T,y}(t,s,x,z)$ then follows.
To see that the constant $C$ above does not depends on $t,T,y$, we have to see that 
$\mu_{\mathcal S} \big( B(x,r)\cap S^{d-1} \big) \le C r^{\gamma-1}$,  where $C$ does not depends on $t,T,y$.
By definition of $\mu_{\mathcal S}$, we have:
$$
\mu_{\mathcal S} \big( B(x,r)\cap S^{d-1} \big) = \int_0^1 \ind_{ \left\{   \frac{ \sigma_v \varsigma}{|\sigma_v \varsigma|} \in B(x,r)\cap S^{d-1} \right\} }|\sigma_v \varsigma|^\alpha \mu(d\varsigma) dv.
$$
Now, observe that $\frac{ \sigma_v \varsigma}{|\sigma_v \varsigma|} \in B(x,r) \Rightarrow
\varsigma \in B( C\sigma_v^{-1}x, Cr)$,
where $C$ is the uniform ellipticity constant of $\sigma(t,x)$.
Consequently, 
$$
\mu_{\mathcal S} \big( B(x,r)\cap S^{d-1} \big) \le \int_0^1 \mu \big( B( C\sigma_v^{-1}x, Cr)\cap S^{d-1} \big) dv \le Cr^{\gamma-1},
$$
uniformly in the parameters  $t,T$ and $y$.
To get a lower bound on $\tilde{p}^{T,y}(t,s,x,z)$, we investigate a lower bound for $p_{\mathcal S}(u,\cdot)$.
To that aim, we prove that when \textbf{[H-LB]} holds for $\nu$, it does for $\nu_{\mathcal S}$.
Specifically, assume \textbf{[H-LB]} holds for $\nu$.
By definition of $\mu_{\mathcal S}$, for all $x\in \R^d$, $r>0$, we have:
$$
\nu_{\mathcal S}\Big( B(x,r)\Big)= \int_0^1 \int_{\R^d} \ind_{\{|x- \sigma_v \varsigma | \le r \}} dv \nu(d\varsigma).
$$
Now, from uniform ellipticity of $\sigma$, 
$$
\{|x- \sigma_v \varsigma | \le r \} \supset \{| \sigma_v^{-1}x- \varsigma | \le Cr \}.
$$

Now, since by assumptions, $\sigma_v^{-1}x \in A_{low}$ for all $v \in [0,1]$, 
we have:
\begin{eqnarray*}
\nu_{\mathcal S}\Big( B(x,r)\Big)= \int_0^1 \nu \Big( \sigma_v^{-1}x, Cr \Big) dv 
\ge \int_0^1 Cr^\gamma \frac{\underline{q}(|\sigma_v^{-1}x|)}{|\sigma_v^{-1}x|^{\alpha +\gamma}}dv \ge Cr^\gamma \frac{\underline{q}(|x|)}{|x|^{\alpha +\gamma}},
\end{eqnarray*}

where to get the last inequality, we exploited the uniform ellipticity of $\sigma$.
Besides, for $r \in (0,1)$, we write using the ellipticity of $\sigma$:
$$
\nu_{\mathcal S}\Big( B(0,r)^c\Big)= \int_0^1 \int_{\R^d} \ind_{\{|\sigma_v \varsigma | \ge r \}} dv \nu(d\varsigma) \le \int_0^1 \int_{\R^d} \ind_{\{| \varsigma | \ge Cr \}} dv \nu(d\varsigma) \le C\frac{1}{r^\alpha}.
$$
Thus, we recovered \textbf{[H-LB]} for $\nu_{\mathcal S}$ and the lower bound holds for $p_{\mathcal S}(u,\cdot)$. Thus, the one for $\tilde{p}^{T,y}(t,s,x,z)$ follows.

\end{proof}

\begin{remark}
The idea of the proof was to identify the density of the frozen process to the density of some L\'evy process and exploit the L\'evy structure to derive bounds on the density.
The procedure described above require the uniform ellipticity of $\sigma$ in order to prove that the assumptions \textbf{[H]} holds for the new L\'evy  process $(\mathcal{S}_u)_{u \ge 0}$.
Intuitively, we can say that a uniform elliptic coefficient does not alter much the nature of the noise in the system.
From the  identity in law \eqref{ID_LOI_MARG} that holds for fixed $s \in [t,T]$ and equation \eqref{FLOW_REVERSE_OK}, we deduce that: 
\begin{equation}\label{REP_DENS_S}
\tilde{p}^{T,y}(t,T,x,y) = p_{\mathcal S}(T-t, \theta_{t,T}(y)-x).
\end{equation}
This identity will be useful when investigating the parametrix kernel $H$.
\end{remark}

Now, we state a Dirac convergence Lemma for the frozen process when the freezing parameter changes.
This convergence will be used in the proof of the well posedness of the martingale problem.
The difficulty comes from the fact that when integrating with respect to the freezing parameter (as it is the case in a parametrix procedure), the Dirac convergence does not follow from the Chapman-Kolmogorov equations.
However, since we have good estimates on the frozen density, we manage to prove the following lemma:

\begin{lemma}\label{convergence_dirac_temp}
For all bounded continuous function $f:\R^{d}\rightarrow \R, x\in \R^{d}$, 

\begin{equation}
\left| \int_{\R^{d}} f(y) \tilde{p}^{T,y}(t,T,x,y)dy -f(x) \right| \underset{T\downarrow t} {\longrightarrow} 0,
\end{equation}
that is, for all $(x,y) \in \R^{d} \times \R^{d}$,  $\tilde p^{T,y}(t,T,x,y)dy \Rightarrow \delta_x(dy)$ weakly when  
$T\downarrow t$.
\end{lemma}

\begin{proof}
We prove this convergence in the lines of \cite{huan:meno:14}.
Let us write:
\begin{eqnarray*}
 \int_{\R^{d}} f(y) \tilde{p}^{T,y}(t,T,x,y)dy -f(x) 
  &=&  \int_{\R^{d}} f(y) \Big(\tilde{p}^{T,y}(t,T,x,y) -  \tilde{p}^{T,\theta_{T,t}(x)}(t,T,x,y)\Big)dy \\
  &&+  \int_{\R^{d}} f(y) \Big(\tilde{p}^{T,\theta_{T,t}(x)}(t,T,x,y) \Big)dy-f(x).
 \end{eqnarray*}
From the usual Dirac convergence in the Kolmogorov equation \eqref{CHAP_KOL_BWD}, the second term tends to zero when $T\rightarrow t$.
We focus on the first term. Define:
\begin{equation}\label{conv_temp}
\Delta = \int_{\R^{d}} f(y) \Big(\tilde{p}^{T,y}(t,T,x,y) -  \tilde{p}^{T,\theta_{T,t}(x)}(t,T,x,y)\Big)dy.
\end{equation}
For a given threshold $K>0$ and a certain (small) $\beta >0$ to be specified, we split $\R^{d}$ into $D_1\cup D_2$ where:
$$
D_1=\left\{ y \in \R^{d}; \frac{|\theta_{t,T}(y)-x|}{(T-t)^{1/\alpha}} \leq K (T-t)^{-\beta} \right\},\ 
D_2=\left\{ y \in \R^{d}; \frac{|\theta_{t,T}(y)-x|}{t^{1/\alpha}} > K (T-t)^{-\beta} \right\}.
$$

A direct application of Proposition \ref{EST_FRZ_DENS_PROP} yields:
$$
\tilde{p}^{T,y}(t,T,x,y) \le C\frac{(T-t)^{-d/\alpha}}{ \left(1+\frac{|\theta_{t,T}(y)-x|}{(T-t)^{1/\alpha}} \right)^{\alpha+\gamma}} \bar{q}(|\theta_{t,T}(y)-x|).
$$

Now, observe that 
$$
y-x-\int_t^TF(u,\theta_{u,T} (\theta_{T,t}(x)))du=y-x-\int_t^TF(u,\theta_{u,t}(x))du=y-\theta_{T,t}(x).
$$ 
Also, from the Lipschitz property of the flow, we have $|\theta_{t,T}(y)-x| \asymp |y-\theta_{T,t}(x)|$.
Consequently, we obtain:
$$
\tilde{p}^{t,\theta_{T,t}(x)}(t,T,x,y) \le C\frac{(T-t)^{-d/\alpha}}{ \left(1+\frac{|\theta_{t,T}(y)-x|}{(T-t)^{1/\alpha}} \right)^{\alpha+\gamma}} \bar{q}(|\theta_{t,T}(y)-x|),
$$
and we have the same upper bound for the two densities in \eqref{conv_temp}.
The idea is that on $D_2$, we use the tail estimate, and on $D_1$, we will explicitly exploit the compatibility between the spectral measures and the Fourier transform in the Fourier representation of the densities.
Set for  $i\in \{1,2\}$: 
$$
\Delta_{D_i}:=\int_{D_i} f(y) \Big(\tilde{p}^{T,y}(t,T,x,y) -  \tilde{p}^{T,\theta_{T,t}(x)}(t,T,x,y)\Big)dy.
$$ 
For $D_2$, we bound the two densities as we described above:
\begin{eqnarray*}
|I_{D_2}|&\leq & C|f|_\infty \int_{D_2} \frac{(T-t)^{-d/\alpha}}{ \left(1+\frac{|\theta_{t,T}(y)-x| }{(T-t)^{1/\alpha}} \right)^{\alpha+\gamma}} \bar{q}(C^{-1}|\theta_{t,T}(y)-x| ) dy \\
&\le& C|f|_\infty \int_{K (T-t)^{-\beta}}^{+\infty} \frac{r^{d-1}}{1+ r^{\alpha+\gamma}}\bar{q}(r)dr\\ 
&\leq &C (T-t)^{\beta(\gamma+\alpha-d)}.
\end{eqnarray*}
Thus, for $\beta>0$, $I_{D_2}\underset{T\downarrow t}{\longrightarrow}0 $.
On $D_1$, we will start from the inverse Fourier representation of $\tilde{p}^{T,z}(t,x,y)$, $z=\theta_{T,t}(x),y$.
Recall we denoted by $\varphi_Z$ the L\'evy Khintchine exponent of $Z$, that is 
$e^{t\varphi_{Z}(p)}=\E(e^{i\langle p,  Z_t \rangle})$, denoting $\sigma^*$ the transpose of $\sigma$,
we have:
\begin{eqnarray*}
\tilde{p}^{T,z}(t,T,x,y)=\frac{1}{(2\pi)^{d}} \int_{\R^{d}} d\zeta e^{-i \langle \zeta ,y-\int_t^T F(u,\theta_{u,T}(z))du-x \rangle} \exp \Big( \int_t^T\varphi_{Z}(\sigma(u,\theta_{u,T}(z))^*\zeta)du\Big).
\end{eqnarray*}

Consequently, we have:
\begin{eqnarray*}
&&\tilde{p}^{T,y}(t,T,x,y) -  \tilde{p}^{T,\theta_{T,t}(x)}(t,T,x,y)\\
 &=& \frac{1}{(2\pi)^d}\int_{\R^d} e^{-i \langle \zeta, y-\int_t^T F(u,\theta_{u,T}(y))du-x \rangle} e^{ \int_t^T \varphi_{Z}(\sigma(u,\theta_{u,T}(y))^*\zeta) du} \\
 &&- e^{-i \langle \zeta, y-\int_t^T F(u,\theta_{u,t}(x))du-x \rangle} e^{ \int_t^T \varphi_{Z}(\sigma(u,\theta_{u,t}(x))^*\zeta)du}d\zeta\\
&=&\frac{1}{(2\pi)^d}\int_{\R^d} \left(e^{-i \langle \zeta, y-\int_t^T F(u,\theta_{u,T}(y))du-x \rangle} - e^{-i \langle \zeta, y-\int_t^T F(u,\theta_{u,t}(x))du-x \rangle}\right) e^{\int_t^T \varphi_{Z}(\sigma(u,\theta_{u,T}(y))^*\zeta) du}d\zeta\\
&&+\frac{1}{(2\pi)^d}\int_{\R^d} e^{-i \langle \zeta, y-\int_t^T F(u,\theta_{u,t}(x))du-x \rangle}\left( e^{ \int_t^T\varphi_{Z}(\sigma(u,\theta_{u,T}(y))^*\zeta)du } -e^{ \int_t^T \varphi_{Z}(\sigma(u,\theta_{u,t}(x))^*\zeta) du}\right) d\zeta \\
&=&\Gamma_1(t,T,x,y)+\Gamma_2(t,T,x,y).
\end{eqnarray*}

We split accordingly:
$$\int_{D_1} f(y) \Big(\tilde{p}^{T,y}(t,T,x,y) -  \tilde{p}^{T,\theta_{T,t}(x)}(t,T,x,y)\Big) dy = \int_{D_1} f(y) \Gamma_1(t,T,x,y) dy +\int_{D_1} f(y) \Gamma_2(t,T,x,y) dy.$$

Note first that when $\alpha \le 1$, we assumed $F=0$, so that the term $\Gamma_1(t,T,x,y) =0$ in that case.
We now treat this term, with $\alpha > 1$.
Using the mean value theorem, we write:
\begin{eqnarray*}
&&\Gamma_1(t,T,x,y)\\
&=&\frac{1}{(2\pi)^d}\int_{\R^d} 
\int_0^1 d\lambda i \langle \zeta, (I-\theta_{T,t})(\theta_{t,T}(y)-x) \rangle e^{-i \langle \zeta, 
[\lambda I +(1-\lambda)\theta_{T,t}](\theta_{t,T}(y)-x) \rangle}
e^{ \int_t^T \varphi_{Z}(\sigma(u, \theta_{u,T}(y))^*\zeta) du} d\zeta,
\end{eqnarray*}
where we denoted by $I$ the identity map of $\R^d$.
Recall that from the Lipschitz property of the flow and Gronwall's Lemma, there exists $C>0$ such that
for all $t \le T$, $z\in \R^d$, $|(I-\theta_{T,t})(z)| \le C(T-t)  (1+|z|)$.
Thus, since $y \in D_1$, we have for $\beta \le 1/\alpha$,
 \begin{eqnarray*}
 |\Gamma_1(t,T,x,y)|&\le& C(T- t) \int_{\R^d}|\zeta| e^{-K(T-t)|\zeta|^\alpha}d\zeta \le C (T-t)^{1-\frac{1}{\alpha}- \frac{d}{\alpha}}.
\end{eqnarray*}
Integrating on $D_1$, we obtain:
\begin{eqnarray*}
\left|\int_{D_1} f(y) \Gamma_1(t,T,x,y) dy\right| \le C|f|_\infty (T-t)^{1-\frac{1}{\alpha} -\beta d}
\underset{T \rightarrow t}{\longrightarrow} 0,
\end{eqnarray*}
when $1/d(1-1/\alpha) >\beta$.
For $\Gamma_2$, we write:

\begin{eqnarray*}
\Gamma_2 (t,T,x,y) &=&
\frac{1}{(2\pi)^d}\int_{\R^d} d\zeta  e^{-i \langle \zeta, y-\theta_{T,t}(x) \rangle} \int_0^1 d\lambda\  
e^{ \int_t^T \lambda\varphi_Z(\sigma(u,\theta_{u,T}(y))^* \zeta) + (1-\lambda)  \varphi_Z(\sigma(u, \theta_{u,t}(x))^* \zeta) du} \\
&&\times \int_t^T( \varphi_Z(\sigma(u,\theta_{u,T}(y))^* \zeta)- \varphi_Z(\sigma(\theta_{u,t}(x))^* \zeta)) du.
\end{eqnarray*}
We know from assumption \textbf{[H-2]} that the L\'evy-Khintchine exponent is bounded by $-K(T-t)|\zeta|^\alpha$, thus, we obtain independently of $\lambda \in (0,1)$:
$$
e^{ \int_t^T \lambda \varphi_Z(\sigma(u,\theta_{u,T}(y))^* \zeta) + (1-\lambda) \varphi_Z(\sigma(u,\theta_{u,t}(x))^* \zeta)du}\le e^{-K(T-t)|\zeta|^\alpha}.
$$
On the other hand, using the bound on the L\'evy-Khintchine exponent and assumption \textbf{[H-5]}, we can rewrite the increment:
\begin{eqnarray*}
&&\left|\int_t^T  \varphi_Z(\sigma(u,\theta_{u,T}(y))^* \zeta)- \varphi_Z(\sigma(u,\theta_{u,t}(x))^* \zeta) du \right|
\\
&=& \left| \int_t^T \int_{\R^{d}} \cos(\langle  \sigma(u,\theta_{u,T}(y))^*\zeta,\xi \rangle)-\cos(\langle\sigma(u,\theta_{u,t}(x))^* \zeta, \xi\rangle) \nu(dz) du \right|\\
&\le& K  |\zeta|^\alpha \int_t^T|\theta_{u,t}(x)-\theta_{u,T}(y)|^{\eta(\alpha\wedge 1)} \le K (T-t)  |\zeta|^\alpha |x-\theta_{t,T}(y)|^{\eta(\alpha\wedge 1)}.
\end{eqnarray*}
To summarize, we obtained:
\begin{eqnarray*}
\int_{D_1} f(y) \Gamma_2(t,T,x,y) &\leq& |f|_\infty\int_{D_1}dy \left| \Gamma_2(t,T,x,y)\right| \\
&\leq& C  \int_{D_1}dy \int_{\R^{d}}    (T-t) |\zeta|^\alpha |x-\theta_{t,T}(y)|^{\eta(\alpha\wedge 1)}
e^{-K(T-t)|\zeta|^\alpha}d\zeta.
\end{eqnarray*}
Changing variables, and integrating over $\zeta$ yields
\begin{eqnarray*}
\int_{D_1} f(y) \Gamma_2(t,T,x,y) &\leq& \frac{C}{t^{d/\alpha}}|f|_\infty
 \int_{D_1}dy |\theta_{t,T}(y)-x|^{\eta(\alpha\wedge 1)} \\
 &=& \frac{C}{(T-t)^{d/\alpha}}|f|_\infty \int_0^{(T-t)^{-\beta}}dr \  r^{\eta(\alpha \wedge 1)+d-1} (T-t)^{d/\alpha+ \eta(1\wedge1/\alpha)}.
\end{eqnarray*}
Choosing now $\frac{\eta(1/\alpha \wedge 1)}{d+\eta(\alpha \wedge 1)} > \beta>0 $ gives that $|I_{D_1}|\underset{T\downarrow t}{\longrightarrow} 0$, which concludes the proof.
\end{proof}

\subsection{The Smoothing Properties of $H(t,T,x,y)$.}
\label{SECT_SMTH_H}

First, we investigate an upper bound for the parametrix kernel. 
Then, we use it to prove a smoothing effect in time, in the sense that the singularity in time is relaxed after an integration in space.
Recall that:
$$
\forall t\ge 0, \ (x,y)\in (\R^{d})^2,\ H(t,T,x,y) := \Big(L(x,\nabla_x) - L(\theta_{t,T}(y),\nabla_x) \Big)\tilde{p}^{T,y}(t,T,x,y).
$$
Heuristically, the difference of the generators should yield a singularity $(T-t)^{-1}$, just as in the Gaussian case.
This singularity is not integrable in time, however, we will prove that after an integration in space (which naturally occurs in the time-space convolution in the parametrix series), it will be.

\begin{proposition}\label{EST_H_PROP}
Assume $\H$ is in force. There exists $C>0$ s.t. for all $t\in (0,T]$, $(x,y)\in (\R^{d})^2 $:
$$
|H(t,T,x,y)| \le  C\left( (T-t)^{-1/\alpha}\ind_{\{\alpha >1\}}+\frac{\delta \wedge |x-\theta_{t,T}(y)|^{\eta(\alpha \wedge 1)}}{T-t} \right) \bar{p}(t,T,x,y)= \bar{H}(t,T,x,y),$$
where
\begin{itemize}

\item when the drift $F$ is bounded, $\theta$ is the identity map: $\theta_{t,T}(y)=y$, and
\begin{itemize}
\item under $\textbf{[H-1a]}$, $\gamma=d$ and for all $s>0$, $Q(s) =\bar{q}(s)$, 
\item under $\textbf{[H-1b]}$, for all $s>0$, $Q(s) =\min(1,s^{\gamma-1})\bar{q}(s)$,
\end{itemize}

\item when the drift $F$ is Lipschitz continuous, $\theta_{s,t}(x)$ denotes the solution to the ordinary differential equation:
$$
\frac{d}{ds} \theta_{s,T}(y) = F(\theta_{s,T}(y)), \ \theta_{T,T}(y) =y, \ \forall 0 \le s \le T,
$$
and 
\begin{itemize}
\item under $\textbf{[H-1a]}$, $\gamma=d$ and for all $s>0$, $Q(s) =\min(1,s)\bar{q}(s)$,
\item under $\textbf{[H-1a]}$, for all $s>0$, $Q(s) =\min(1,s,s^{\gamma-1})\bar{q}(s)$.
\end{itemize}

\end{itemize}
\end{proposition}
Thus, the upper bound on the Kernel $H$ is the same as the upper bound on the Frozen density $\tilde p^{T,y}(t,T,x,y)$ up to the additional multiplier $ (T-t)^{-1/\alpha} \ind_{\alpha >1}+ \big(\delta \wedge |x-\theta_{t,T}(y)|^{\eta(\alpha \wedge 1)} \big)(T-t)^{-1}$, that can be seen as the singularity induced by the difference $L(x,\nabla_x) - L(\theta_{t,T}(y),\nabla_x)$ applied to the frozen density.
The proof proceeds following the lines of Sztonyk \cite{szto:10}, splitting the large jumps and the small jumps.
The small jumps are dealt using Fourier analysis techniques, whereas the big jumps are dealt more directly.
\begin{proof}

Recall that from \eqref{ID_LOI_MARG}, the density of $\tilde{X}^{T,y}_s$ can be linked to the density of the L\'evy process 
$(\mathcal{S}_u)_{u \ge 0}$ considered at time $s-t$.
We now exploit the L\'evy structure of $(\mathcal{S}_u)_{u \ge 0}$ to obtain an upper bound on $H(t,T,x,y)$.
Specifically, let us introduce the L\'evy-It\^o decomposition of $(\mathcal{S}_u)_{u \ge 0}$:
$$
\mathcal{S}_u= M_u+N_u,
$$
where $(M_u)_{u \ge 0}$ is a martingale and $(N_u)_{u\ge 0}$ is a compound Poisson process.
We choose to place the cut-off at the characteristic time-scale, namely $(T-t)^{1/\alpha}$.
Therefore, the Fourier transform of $M_u$ writes:

$$
\E(e^{i\langle \zeta, M_u \rangle}) = \exp \left( u \int_{\R^d} (e^{i \langle \zeta, \eta \rangle} - 1 - i\langle \zeta ,\eta \rangle) 
\ind_{\{ |\eta|\le (T-t)^{1/\alpha}\}} \nu_\mathcal{S}(d\eta) \right).
$$ 
This expression is integrable and regular in the variable $\zeta$ (see Section 2 in Sztonyk \cite{szto:10} and the references therein).
Thus, the density $p_M(u,\cdot)$ of $M_u$ exists and is the  Schwartz's class.
Thus, we can say that this term produces the density in the L\'evy-It\^o decomposition.
Also, denoting by $\bar{\nu}_{\mathcal S}(dz) = \ind_{\{ |z| \ge (T-t)^{1/\alpha}\}} \nu_{\mathcal S}(dz)$, we have the following decomposition for the law of the compound Poisson process $N_u$:
$$
P_{N_u}(dz) = e^{-u\bar{\nu}_{\mathcal{S}}(\R^d)} \sum_{k=0}^{+\infty} \frac{u^k \bar{\nu}_{\mathcal{S}}^{*k}(dz)}{k!}. 
$$

Now, by independence of $(M_u)_{u \ge 0}$ and $(N_u)_{u\ge 0}$ and exploiting equation \eqref{REP_DENS_S}, we get:
\begin{equation}\label{DESINTEGRATION_DENS_LEVY_ITO}
\tilde{p}^{T,y}(t,T,x,y) =p_\mathcal{S}(T-t, \theta_{t,T}(y) - x)= \int_{\R^d} p_M(T-t, \theta_{t,T}(y) - x - \xi ) P_{N_{T-t}}(d\xi).
\end{equation}

From the definition of the generators, the operator naturally splits into three parts,
for a test function $\varphi$:
\begin{eqnarray*}
\Big(L_t(x,\nabla_x) - L_t(\theta_{t,T}(y),\nabla_x) \Big)\varphi(x) = \langle \nabla_x \varphi(x), F(t,x)-F(t,\theta_{t,T}(y)) \rangle\\
+\int_{\R^d} \Big(\varphi(x+z) - \varphi(x) - \langle \nabla \varphi(x),z\rangle \Big) \ind_{\{|z|\le (T-t)^{1/\alpha} \}}(\nu_t(x,dz)-\nu_t(\theta_{t,T}(y),dz))\\
+\int_{\R^d} \Big( \varphi(x+z) - \varphi(x) \Big)\ind_{\{|z|\ge (T-t)^{1/\alpha} \}}(\nu_t(x,dz)-\nu(\theta_{t,T}(y),dz)).
\end{eqnarray*}
Recall that we defined $\nu_t(\xi,A)= \nu\{z\in \R^d; \sigma(t,\xi)z \in A \}.$
Also, observe that by symmetry of $\nu$, we changed the cut-off function to exhibit the intrinsic time-scale.
Note that the first order term in the operator is present only in the case $\alpha > 1$.
Otherwise, we assumed that $F=0$.

Thus, a derivative along $x$ of $\tilde{p}(t,T,x,y)$ acts in fact on the density of the martingale, and we have to control  $\nabla_x p_M(T-t, \theta_{t,T}(y) - x - \xi )$.
Borrowing the notations of the proof of Lemma 2 in \cite{szto:10}, we have:
$$
p_M(T-t, \theta_{t,T}(y) - x - \xi) = (T-t)^{-d/\alpha} g_{T-t}( (T-t)^{-1/\alpha} \theta_{t,T}(y) - x - \xi).
$$
Formally, since we chose to split at the characteristic time-scale $(T-t)^{1/\alpha}$, the density of the martingale presents a time space separation, and defining $g_{T-t}$ as above allows to have estimates independent of $T-t$.
Thus, uniformly for all $T-t>0$, $g_{T-t}(\cdot)$ is in Schwartz's class.
Therefore, we have for all $m\ge 1$:
$$
|\nabla_x p_M(T-t, \theta_{t,T}(y) - x - \xi)| \le  \frac{1}{(T-t)^{1/\alpha}} C_m (T-t)^{-d/\alpha} \left( 1+\frac{|\theta_{t,T}(y) - x - \xi|}{(T-t)^{1/\alpha}}\right)^{-m},
$$
and we recovered Lemma 2 in \cite{szto:10}, up to the singularity $1/(T-t)$.
Consequently, integrating this estimate against the law of the large jumps, we obtain the following estimate on the gradient of the density:
$$
\nabla_x \tilde{p}(t,T,x,y) \le \frac{C}{T-t}  \frac{(T-t)^{-d/\alpha}}{ \left(1+\frac{|\theta_{t,T}(y)-x|}{(T-t)^{1/\alpha}} \right)^{\alpha+\gamma}}  \bar{q}(|\theta_{t,T}(y)-x|) 
$$

Thus, a derivative on the density yields a singularity in $t^{-1/\alpha}$ which is integrable when $\alpha >1$.
Specifically, when the drift $F$ is bounded and $\alpha>1$, we write:

\begin{eqnarray*}
|\langle \nabla_x p(t,T,x,y), F(x)-F(\theta_{t,T}(y)) \rangle| \le  C2|F|_{\infty} |\nabla_x p(t,T,x,y)| \\
\le C (T-t)^{-1/\alpha} \frac{(T-t)^{-d/\alpha}}{ \left(1+\frac{|\theta_{t,T}(y)-x|}{(T-t)^{1/\alpha}} \right)^{\alpha+\gamma}}  \bar{q}(|\theta_{t,T}(y)-x|) .
\end{eqnarray*}

On the other hand, when $F$ is unbounded, we have to deteriorate the tempering function, exploiting the Lipschitz property of $F$:

\begin{eqnarray*}
|\langle \nabla_x p(t,T,x,y), F(x)-F(\theta_{t,T}(y)) \rangle| &\le&  C|x-\theta_{t,T}(y)| |\nabla_x p(t,T,x,y)| \\
&\le& C (T-t)^{-1/\alpha} \frac{(T-t)^{-d/\alpha}}{ \left(1+\frac{|\theta_{t,T}(y)-x|}{(T-t)^{1/\alpha}} \right)^{\alpha+\gamma}} |x-\theta_{t,T}(y)| \bar{q}(|\theta_{t,T}(y)-x|) \\
&\le& C (T-t)^{-1/\alpha} \frac{(T-t)^{-d/\alpha}}{ \left(1+\frac{|\theta_{t,T}(y)-x|}{(T-t)^{1/\alpha}} \right)^{\alpha+\gamma}}  Q(|\theta_{t,T}(y)-x|).
\end{eqnarray*}

\begin{remark}
Let us mention here the works of Knopova and Kulik \cite{knop:kuli:14}, where the authors prove that the gradient of the density of a rotationally invariant stable process actually have the estimate:
$$
p_Z(1,x)\le \frac{1}{|x|^{d+\alpha+1}}.
$$
In other words, the exponent is actually bigger in this case. This fact allows to compensate the singularity of the gradient part when $\alpha \le 1$.
This very important result is very specific to the case of the rotationally  invariant stable process, as it is the particular form of the L\'evy measure that allows the growth of the exponent.
In our general setting, we have been unable to prove a similar estimate, and had to restrict ourselves to remove the drift when $\alpha \le 1$.
Observe nonetheless that this condition is relatively classic in the literature dealing with general stable process (see e.g. Kolokolstov \cite{kolo:00}).
\end{remark}

Consider now the integro-differential part of the kernel.
For the small jumps part, once again, we observe that the operator acts on the variable $x$, and thus can be put on the density of the martingale.
We use the representation in terms of symbols, denoting by $\phi_t(x,p)$ the symbol of an integro-differential operator $\Phi_t(x,\nabla_x)$:
\begin{eqnarray*}
&&\Phi_t(x,\nabla_x) p_M\Big(T-t, \sigma(\theta_{t,T}(y))^{-1}(\theta_{t,T}(y)-x)-\xi \Big)\\
& = &
\frac{1}{(2\pi)^d} \int_{\R^{d}} e^{-i \langle \zeta, \sigma(\theta_{t,T}(y))^{-1}(\theta_{t,T}(y)-x) - \xi \rangle} \phi_t(x, -(\sigma(\theta_{t,T}(y))^{-1})^* \zeta(T-t)^{-1/\alpha}) \hat{g}_{T-t}(\zeta) d\zeta.
\end{eqnarray*}

Now, when $\Phi_t(x,\nabla_x)= L_t^M(x,\nabla_x) - L_t^M(\theta_{t,T}(y),\nabla_x)$, the small jump part of the difference of the generators, that is:
\begin{eqnarray*}
&&\Big(L_t^M(x,\nabla_x) - L_t^M(\theta_{t,T}(y),\nabla_x)\Big)\varphi(x)\\
&=&\int_{\R^d} \Big(\varphi(x+z) - \varphi(x) - \langle \nabla \varphi(x),z\rangle \Big) \ind_{\{|z|\le (T-t)^{1/\alpha} \}}(\nu(x,dz)-\nu(\theta_{t,T}(y),dz)),
\end{eqnarray*}
 denoting by $l_t^M(x,\zeta) - l_t^M(\theta_{t,T}(y),\zeta)$ the corresponding symbol, we have that:
$$
|l_t^M(x,\zeta) - l_t^M(\theta_{t,T}(y),\zeta)| \le C \delta \wedge |\theta_{t,T}(y)-x|^{\eta(\alpha\wedge1)} |\zeta|^\alpha.
$$
Moreover, this quantity $\phi_t(x, -(\sigma(\theta_{t,T}(y))^{-1})^* \zeta(T-t)^{-1/\alpha}) \hat{g}_{T-t}(\zeta)$ is smooth (in its $\zeta$ argument) because of the truncation (see Theorem 3.7.13 in Jacob \cite{jacob1}).
Consequently, 
$$
\frac{T-t}{\delta\wedge |\theta_{t,T}(y)-x|^{\eta(\alpha\wedge1)}}\phi_t(x, -(\sigma(\theta_{t,T}(y))^{-1})^* \zeta(T-t)^{-1/\alpha}) \hat{g}_{T-t}(\zeta)
$$
is infinitely differentiable as a function of $\zeta$ and uniformly bounded with all its derivatives.
Therefore, it is in Schwartz's space as well as its Fourier inverse. We have $\forall m>1$:
\begin{eqnarray*}
&&\left|(L_t^M(x,\nabla_x) - L_t^M(\theta_{t,T}(y),\nabla_x))p_M\Big(T-t, \sigma(\theta_{t,T}(y))^{-1}(\theta_{t,T}(y)-x)-\xi \Big)\right|\\
 &&\le C \frac{\delta\wedge |\theta_{t,T}(y)-x|^{\eta(\alpha\wedge1)}}{T-t} (T-t)^{-d/\alpha} \left(1 + \frac{|\sigma(\theta_{t,T}(y))^{-1}(\theta_{t,T}(y)-x)-\xi|}{(T-t)^{1/\alpha}} \right)^{-m}.
\end{eqnarray*}

Consequently, we recovered Lemma 2 in \cite{szto:10} for the parametrix kernel, up to the additional multiplicative term
$\big(\delta\wedge |\theta_{t,T}(y)-x|^{\eta(\alpha\wedge1)} \big)(T-t)^{-1}$, which is the expected singularity for the kernel (see Kolokoltsov \cite{kolo:00}).
The upper bound follows from this upper bound and the control of the measure of the balls for $P_{N_{T-t}}$  similarly to the derivation of the upper bound for the density, see Corollary 6 in \cite{szto:10} and the proof of Theorem 1 in  \cite{szto:10}. 
The upper bound for the small jumps part of the kernel follows.

Finally, for large jumps, we see 
that the measure $\ind_{\{|\xi|\ge (T-t)^{1/\alpha} \}}(\nu(x,d\xi)-\nu(\theta_{t,T}(y),d\xi))$ is no more singular.
Thus, we can write:
\begin{eqnarray*}
&&\left|\int_{\R^d} \Big( \tilde{p}(t,T,x+\xi,y) -\tilde{p}(t,T,x,y) \Big)\ind_{\{|\xi|\ge (T-t)^{1/\alpha} \}}(\nu(x,d\xi)-\nu(\theta_{t,T}(y),d\xi))\right|\\
&&\le \int_{\R^d} \Big| \tilde{p}(t,T,x+\xi,y) -\tilde{p}(t,T,x,y) \Big|\ind_{\{|\xi|\ge (T-t)^{1/\alpha} \}}|\nu(x,d\xi)-\nu(\theta_{t,T}(y),d\xi)|\\
&&\le  
\delta\wedge|x-\theta_{t,T}(y)|^{\eta(\alpha \wedge 1)} \bigg(\int_{S^{d-1}}\int_0^{+\infty}  \tilde{p}(t,T,x+s\varsigma,y) \ind_{\{s\ge (T-t)^{1/\alpha} \}} \frac{\bar{q}(s)}{s^{1+\alpha}}ds \mu(d\varsigma) \\
&& \qquad+ \frac{1}{T-t}\tilde{p}(t,T,x,y) \bigg).
\end{eqnarray*}
For the last inequality, we exploited \textbf{[H-5]}.
We focus on the remaining integral term above.
When the diagonal regime holds, the estimate is straightforward, as we can directly bound $ \tilde{p}(t,T,x+s\varsigma,y)\le C(T-t)^{-d/\alpha}\le C \bar{p}(t,T,x,y)$. The integral then yields the singularity $(T-t)^{-1}$.
Therefore, we assume that $|\theta_{t,T}(y)-x| \ge (T-t)^{1/\alpha}$.
The regime of $\tilde{p}(t,T,x+s\varsigma,y) $ is given by $|\theta_{t,T}(y)-x-s\varsigma|$.
Thus, thanks to the triangle inequality, when $|\theta_{t,T}(y)-x|\le 1/2 s$, or when $s \le 1/2|\theta_{t,T}(y)-x|$,
the density $\tilde{p}(t,T,x+s\varsigma,y) $ is off-diagonal with $\tilde{p}(t,T,x+s\varsigma,y) \le C \bar p(t,x,y)$.

Consequently, the problematic case is when $s\asymp |\theta_{t,T}(y)-x|$. 
Indeed, in this case, $\tilde{p}(t,T,x+s\varsigma,y) $ can be in diagonal regime, whereas $\tilde{p}(t,T,x,y)$ is still in the off-diagonal regime.

Assume first that \textbf{[H-1-a]} holds, and let us simply denote $\frac{d\mu}{d\varsigma}(\varsigma)$ the density of $\mu$ on the sphere.
Then, we have:
\begin{eqnarray*}
&&\int_{1/2|\theta_{t,T}(y)-x|}^{3/2|\theta_{t,T}(y)-x|}  \int_{S^{d-1}}  \tilde{p}(t,T,x+s\varsigma,y) \ind_{\{s\ge (T-t)^{1/\alpha} \}} \frac{\bar{q}(s)}{s^{1+\alpha}}ds \frac{d\mu}{d\varsigma}(\varsigma)d\varsigma\\
&=&\int_{1/2|\theta_{t,T}(y)-x|}^{3/2|\theta_{t,T}(y)-x|}  \int_{S^{d-1}}  \tilde{p}(t,T,x+s\varsigma,y) \ind_{\{s\ge (T-t)^{1/\alpha} \}} \frac{\bar{q}(s)}{s^{\alpha+d}} \frac{d\mu}{d\varsigma}(\varsigma) s^{d-1}dsd\varsigma.
\end{eqnarray*}
Now, since $s\asymp |\theta_{t,T}(y)-x|$, we can take $\frac{\bar{q}(s)}{s^{\alpha+d}}$ out of the integral.
Also, the density $\frac{d\mu}{d\varsigma}(\varsigma)$ is bounded, so that we obtain:

\begin{eqnarray*}
&&\int_{1/2|\theta_{t,T}(y)-x|}^{3/2|\theta_{t,T}(y)-x|}  \int_{S^{d-1}}  \tilde{p}(t,T,x+s\varsigma,y) \ind_{\{s\ge (T-t)^{1/\alpha} \}} \frac{\bar{q}(s)}{s^{1+\alpha}}ds \frac{d\mu}{d\varsigma}(\varsigma)d\varsigma\\
&\le& C\frac{\bar{q}(|\theta_{t,T}(y)-x|)}{|\theta_{t,T}(y)-x|^{\alpha+d}}\int_0^{+\infty} \int_{S^{d-1}}  \tilde{p}(t,T,x+s\varsigma,z) \ind_{\{s\ge (T-t)^{1/\alpha} \}} s^{d-1}dsd\varsigma.
\end{eqnarray*}
Finally, the remaining integral can be bounded by some constant as the integral of the density .
Consequently, we obtained:

\begin{eqnarray*}
 \int_{1/2|\theta_{t,T}(y)-x|}^{3/2|\theta_{t,T}(y)-x|}  \int_{S^{d-1}}  \tilde{p}(t,T,x+s\varsigma,y) \ind_{\{s\ge (T-t)^{1/\alpha} \}} \frac{\bar{q}(s)}{s^{1+\alpha}}ds \frac{d\mu}{d\varsigma}(\varsigma)d\varsigma \\
 \le C \frac{\bar{q}(|\theta_{t,T}(y)-x|)}{|\theta_{t,T}(y)-x|^{\alpha+d}} = C\frac{1}{T-t} \frac{T-t}{|\theta_{t,T}(y)-x|^{\alpha+d}} \bar{q}(|\theta_{t,T}(y)-x|),
\end{eqnarray*}
which is the off diagonal estimate for $\bar{p}$ when \textbf{[H-1a]} holds, up to the singularity $1/(T-t)$.

Now, assume that \textbf{[H-1-b]} holds.
In this case, we can take out $\frac{\bar{q}(s)}{s^{1+\alpha}}$ and integrate a density to get:
\begin{eqnarray*}
\int_{1/2|\theta_{t,T}(y)-x|}^{3/2|\theta_{t,T}(y)-x|}  \int_{S^{d-1}}  \tilde{p}(t,T,x+s\varsigma,y) \ind_{\{s\ge (T-t)^{1/\alpha} \}} \frac{\bar{q}(s)}{s^{1+\alpha}}ds \mu(d\varsigma)\\
%
\le \frac{\bar{q}(|\theta_{t,T}(y)-x|)}{|\theta_{t,T}(y)-x|^{1+\alpha}}\int_0^{+\infty} \int_{S^{d-1}}  \tilde{p}^y(t,T,x+s\varsigma,z) \ind_{\{s\ge (T-t)^{1/\alpha} \}} ds \mu(d\varsigma) \\
\le C \frac{\bar{q}(|\theta_{t,T}(y)-x|)}{|\theta_{t,T}(y)-x|^{1+\alpha}}.
\end{eqnarray*}

Rewriting the right hand side to make the time dependencies  appear : 
\begin{eqnarray*}
\frac{\bar{q}(|\theta_{t,T}(y)-x|)}{|\theta_{t,T}(y)-x|^{1+\alpha}}&=&\frac{1}{T-t}\frac{(T-t)^{1+\frac{\gamma-d}{\alpha}}}{|\theta_{t,T}(y)-x|^{1+\alpha}}\bar{q}(|\theta_{t,T}(y)-x|) \times (T-t)^{\frac{d-\gamma}{\alpha}}\\
 &\le& C \frac{1}{T-t}\frac{(T-t)^{1+\frac{\gamma-d}{\alpha}}}{|\theta_{t,T}(y)-x|^{1+\alpha}}\bar{q}(|\theta_{t,T}(y)-x|).
\end{eqnarray*}
In the last inequality, we recall that $\gamma \le d$, so that $(T-t)^{\frac{d-\gamma}{\alpha}} \le 1$.
Now, we write:
$$
\frac{(T-t)^{1+\frac{\gamma-d}{\alpha}}}{|\theta_{t,T}(y)-x|^{1+\alpha}}\bar{q}(|\theta_{t,T}(y)-x|)
=\frac{(T-t)^{1+\frac{\gamma-d}{\alpha}}}{|\theta_{t,T}(y)-x|^{\alpha+\gamma}} \times |\theta_{t,T}(y)-x|^{\gamma-1}\bar{q}(|\theta_{t,T}(y)-x|).
$$
Recalling we denoted by $Q$:
$$Q(|\theta_{t,T}(y)-x|) = \max(1, |\theta_{t,T}(y)-x|,|\theta_{t,T}(y)-x|^{\gamma-1})\bar{q}(|\theta_{t,T}(y)-x|),$$

we finally obtain:
$$
C \frac{1}{T-t}\frac{(T-t)^{1+\frac{\gamma-d}{\alpha}}}{|\theta_{t,T}(y)-x|^{1+\alpha}}\bar{q}(|\theta_{t,T}(y)-x|) \le C \frac{1}{T-t}\frac{(T-t)^{1+\frac{\gamma-d}{\alpha}}}{|\theta_{t,T}(y)-x|^{\alpha+\gamma}}Q(|\theta_{t,T}(y)-x|).
$$

%
In other words, we can correct the wrong decay by deteriorating the temperation.
Consequently, the global upper bound for the kernel is the one announced.

To sum up, we deteriorated the tempering function in the following cases:
\begin{itemize}
\item when the drift is bounded and \textbf{[H-1b]} holds.
In this case, we replaced $\bar{q}$ by $Q(s) = \max(1,s^{\gamma-1})\bar{q}(s)$.
\item when the drift is unbounded and \textbf{[H-1a]} holds.
In this case, we replaced $\bar{q}$ by $Q(s) = \max(1,s)\bar{q}(s)$.
\item when the drift is unbounded and \textbf{[H-1b]} holds.
In this case, we replaced $\bar{q}$ by $Q(s) = \max(1,s,s^{\gamma-1})\bar{q}(s)$.
\end{itemize}
Note that when the drift is bounded and \textbf{[H-1a]} holds, we do not need to deteriorate the tempering function.

\end{proof}

\begin{remark}\label{ROT_INV_STABLE_OK}

In the above proof, the temperation only serves to compensate the bad concentration in the generator.
Also, we see that when the spectral measure $\mu$ dominating the L\'evy measure $\nu$ has a density on the sphere, then, the large jump part of the difference of the generators becomes:
$$
\int_{\R^d}  \tilde{p}(t,T,x+\xi,y) \ind_{\{|\xi|\ge (T-t)^{1/\alpha} \}}\nu(d\xi) 
\le
C\int_{\R^d}\tilde{p}(t,T,x+\xi,y) \ind_{\{|\xi|\ge (T-t)^{1/\alpha} \}} \frac{\bar{q}(|\xi|)}{|\xi|^{d+\alpha}}d\xi.
$$
Thus, when $s\asymp |\theta_{t,T}(y)-x|$, as in the last case discussed above, we have directly the good concentration index and the temperation is not needed.
In particular, when $\bar{q}=1$, we recovered results in Kolokolstov \cite{kolo:00}.
\end{remark}

We have obtained the same type of estimate on the kernel and on the frozen density.
Let us observe that the upper bound satisfies a "semi group" property in the following sense.

\begin{lemma}\label{SEMI_GROUP_PROP}
Fix $\tau\in [t,T]$. Let us denote 
$$
\bar{p}_C(t,T,x,y) = \frac{(T-t)^{-d/\alpha}}{\left( 1+ \frac{|\theta_{t,T}(y)-x|}{(T-t)^{1/\alpha}}\right)^{\alpha+\gamma}} Q(C|\theta_{t,T}(y)-x|).
$$

Let $C_1,C_2>0$. For all $\tau \in [t,T]$, there exists $C_3>0$:
$$
\int_{\R^d} \bar{p}_{C_1}(t,\tau,x,z) \bar{p}_{C_2}(\tau,T ,z,y) dz \le C \bar{p}_{C_3}(t,T,x,y).
$$
\end{lemma}

\begin{proof}
From the triangle inequality, we can write:
\begin{eqnarray*}
|\theta_{t,T}(y)-x| &\le& |\theta_{t,T}(y)-\theta_{t,\tau}(z)| + |\theta_{t,\tau}(z)-x|\\
&\le&| \theta_{t,\tau}(\theta_{\tau,T}(y))-\theta_{t,\tau}(z) )| + |\theta_{t,\tau}(z)-x|\\
&\le& | \theta_{\tau,T}(y)- z| + |\theta_{t,\tau}(z)-x|.
\end{eqnarray*}
The last inequality follows from the Lipschitz property of the flow.
This means that we have either $| \theta_{\tau,T}(y)- z| \ge C |\theta_{t,T}(y)-x|$ or $|\theta_{t,\tau}(z)-x|\ge C |\theta_{t,T}(y)-x|$.
Now, since for all $t,T\in \R_+$, $u\mapsto \frac{(T-t)^{-d/\alpha}}{\left( 1+ \frac{u}{(T-t)^{1/\alpha}}\right)^{\alpha+\gamma}} Q(Cu)$ is decreasing, we have either:
$$
\bar{p}_{C_1}(t,\tau,x,z) \le C  \bar{p}_{C_3}(t,T,x,y) \mbox{ or }  \bar{p}_{C_2}(\tau,T ,z,y) \le C  \bar{p}_{C_3}(t,T,x,y).
$$

Taking the corresponding density out of the integral, we can integrate the remaining density to one.
The conclusion of the Lemma follows.

\end{proof}

We exhibit here some smoothing properties in time of the parametrix kernel.
These properties will become crucial when investigating the convergence of the series \eqref{SERIE_PARAM_TEMP} on the one hand and the lower bound of Theorem \ref{MTHM_TEMP} on the other.
The following lemma is a regularizing effect in time of the parametrix kernel.
\begin{lemma}\label{SMTH_H}
There exists $C>1$, $\omega>0$ s.t. for all $\tau\in (t,T)$, $(x,y)\in (\R^{d})^2 $:
\begin{eqnarray*}
\int_{\R^d} \delta \wedge |x-\theta_{\tau,t}(z)|^{\eta(\alpha\wedge 1)} \bar{p}(t,\tau,x,z) dz 
&\le& C (\tau-t)^{\omega}, \\ \int_{\R^d} \delta \wedge |\theta_{\tau,T}(y)-z|^{\eta(\alpha\wedge 1)} \bar{p}(\tau,T,z,y) dz 
&\le& C (T-\tau)^{\omega}.
\end{eqnarray*}
As a corollary, we get that 
$$
\int_{t}^T d\tau\int_{\R^d} |H(\tau,T,z,y)|dz \le C[(T-t)^{\omega} +(T-t)^{1-1/\alpha}\ind_{\alpha>1} ].
$$
Thus, when integrated in time, the parametrix Kernel yields has a smoothing property in time.
\end{lemma}

\begin{proof}
The two estimates are similar, we shall only prove one.
Let us denote the quantity of interest:
$$
I= \int_{\R^d}dz \ \delta \wedge |x-\theta_{t,\tau}(z)|^{\eta(\alpha\wedge 1)} \frac{(\tau-t)^{-d/\alpha}}{ \left(1+\frac{|x-\theta_{t,\tau}(z)|}{(\tau-t)^{1/\alpha}} \right)^{\alpha+\gamma}} \bar{q}(|x-\theta_{t,\tau}(z)|).
$$
We split $\R^d=D_1 \cup D_2$, with
\begin{eqnarray*}
D_1&=&\{z \in \R^d ; |x-\theta_{t,\tau}(z)|\le C (\tau-t)^{1/\alpha}\} \\
D_2&=& \{z \in \R^d ; |x-\theta_{t,\tau}(z)|> C (\tau-t)^{1/\alpha} \}.
\end{eqnarray*}
We write $I_{D_i}$ for the integral over $z\in D_i$. For $z\in D_1$ we have:
$$
 \frac{(\tau-t)^{-d/\alpha}}{ \left(1+\frac{|x-\theta_{t,\tau}(z)|}{(\tau-t)^{1/\alpha}} \right)^{\alpha+\gamma}} \bar{q}(|x-\theta_{t,\tau}(z)|) \le (\tau-t)^{-d/\alpha}, \ |x-\theta_{t,\tau}(z)|^{\eta(\alpha\wedge 1)}\le (\tau-t)^{\eta(1\wedge 1/\alpha)}.
$$
Also, $D_1$ is a compact and its Lebesgue measure is exactly $(\tau-t)^{d/\alpha}$, thus, we obtain
$I_{D_1}\le (\tau-t)^{\eta(1\wedge 1/\alpha)}$.

When $z\in D_2$, we have:
\begin{eqnarray*}
I_{D_2} &\le&  \int_{D_2}dz \delta \wedge |x-\theta_{t,\tau}(z)|^{\eta(\alpha\wedge 1)} \frac{(\tau-t)^{1+\frac{\gamma-d}{\alpha}}}{|x-\theta_{t,\tau}(z)|^{\alpha+\gamma}} \\
&\le&(\tau-t)^{1+\frac{\gamma-d}{\alpha}} \int_{|z- \theta_{\tau,t}(x)|> C (\tau-t)^{1/\alpha}} dz\frac{\delta \wedge |z- \theta_{\tau,t}(x)|^{\eta(\alpha\wedge 1)} }{|z- \theta_{\tau,t}(x)|^{\alpha+\gamma}}.
\end{eqnarray*}
Observe that we used the Lipschitz property of the flow to switch from $x-\theta_{t,\tau}(z)$ to $z- \theta_{\tau,t}(x)$.
This allows us to change variables and set $X=(z- \theta_{\tau,t}(x))/(\tau-t)^{1/\alpha}$, we get:
$$
I_{D_2}\le(\tau-t)^{1+\frac{\gamma-d}{\alpha}} \int_{|X|>1} \frac{(\tau-t)^{\eta(1\wedge \frac{1}{\alpha})}|X|^{\eta(\alpha\wedge 1)}}{|X|^{\alpha+\gamma}} dX .
$$
Thus, the result follows when $\alpha+\gamma-d > \eta(\alpha\wedge 1)$.
When it is not the case, we split again:
\begin{eqnarray*}
 \int_{|z- \theta_{\tau,t}(x)|>(\tau-t)^{1/\alpha}} \frac{\delta \wedge |z- \theta_{\tau,t}(x)|^{\eta(\alpha\wedge 1)} }{|z- \theta_{\tau,t}(x)|^{\alpha+\gamma}} dz
 = \int_{1 \ge |z- \theta_{\tau,t}(x)|>(\tau-t)^{1/\alpha}} \frac{\delta \wedge |z- \theta_{\tau,t}(x)|^{\eta(\alpha\wedge 1)} }{|z- \theta_{\tau,t}(x)|^{\alpha+\gamma}}dz \\
 + \int_{|z- \theta_{\tau,t}(x)|>1} \frac{\delta \wedge |z- \theta_{\tau,t}(x)|^{\eta(\alpha\wedge 1)} }{|z- \theta_{\tau,t}(x)|^{\alpha+\gamma}}dz .
\end{eqnarray*}
The second part of the right hand side is clearly a constant, bounding $\delta \wedge |z- \theta_{\tau,t}(x)|^{\eta(\alpha\wedge 1)}\le \delta$, since $\alpha+\gamma>d$.
For the first part, we change variable again to $Y=(z- \theta_{\tau,t}(x))$, which yields when $\alpha+\gamma-d < \eta(\alpha\wedge 1)$:
$$
\int_{1>|Y|>(\tau-t)^{1/\alpha}} \frac{|Y|^{\eta(\alpha\wedge 1)} }{|Y|^{\alpha+\gamma}} dY \le C.
$$

On the other hand, when $\alpha+\gamma-d = \eta(\alpha\wedge 1)$
$$
\int_{1>|Y|>(\tau-t)^{1/\alpha}} \frac{1}{|Y|^{d}} dY = [\log (|Y|)]_{(\tau-t)^{1/\alpha}}^{1} \le \frac{1}{\alpha}|\log(\tau-t)|.
$$
Thus the proof is complete.

\end{proof}

\subsection{Proof of Theorem \ref{MART_PB_THM_TEMP}: Uniqueness to the Martingale Problem}
\label{PROOF_MART_PB}
We are now in position to prove the uniqueness to the martingale problem.
Our approach relies on the smoothing properties of the parametrix kernel $H$.
This method for proving the uniqueness to the martingale problem has initially been developed by Bass and Perkins \cite{bass:perk:09}, and adapted to the parametrix setting by Menozzi \cite{meno:11}.
For the sake of consistency, we choose to keep the notations in Menozzi \cite{meno:11}, as the proof is essentially the same.

\begin{proof}
We focus on uniqueness.
The existence can be derived from standard compactness arguments (see e.g. Chapter 6 in Stroock and Varadhan \cite{stro:vara:79}, or Stroock \cite{stro:75}).
Suppose we are given two solutions $\mathbb \PP^1$ and $\mathbb \PP^2$ of the martingale problem associated with $L(\cdot, \nabla_\cdot)$, 
starting in $x$ at time $0$. We can assume w.l.o.g. that $t\le T$, the fixed time horizon.
Define for a bounded Borel function $f:[0,T]\times \R^{d}\rightarrow \R$, 
$$S^i f = \E^i \left( \int_t^T f(s,X_s) ds\right),\ i\in \{1,2\},$$
where $(X_t)_{t\ge 0}$ stands for the canonical process associated with $(\PP^i)_{i\in \{ 1,2\}} $. Let us specify that $S^if$ is \textit{a priori} only a linear functional and not a function since $\PP^i$ does not need to come from a Markov process. We denote: 
$$S^\Delta f = S^1f-S^2f,$$ 
and the aim of this section is to prove that $S^\Delta f=0$ for $f$ in a suitable class of test functions.
Since this functional characterises the law of the process, we can conclude to the uniqueness.

If $f \in \mathcal C^{1,2}_0([0,T)\times \R^{d},\R)$, since $(\PP^i)_{i\in \{1,2\}}$ both solve the martingale problem, we have:
\begin{equation}
f(t,x) + \E^i\left(\int_t^T (\partial_s+L_s(x,\nabla_x) f(s,X_s) ds \right) =0,\ i\in\{ 1,2\}.
\end{equation}
As a consequence we thus have that for all $f\in C^{1,2}_0([0,T]\times \R^{nd},\R)$, 

\begin{eqnarray} \label{+-GENE GEL}
S^\Delta \Big( (\partial_s+L_s(x,\nabla_x)) f \Big)=0.
\end{eqnarray}

We now want to apply \eqref{+-GENE GEL} to a suitable function $f$.
For a fixed point $y\in \R^{d}$ and a given $\varepsilon\ge 0$, introduce for all $ f\in \mathcal C^{1,2}_0([0,T)\times \R^{d},\R)$ the operator:
$$
\forall (t,x) \in [0,T)\times\R^{d},  G^{\varepsilon,y}f(t,x) = 
\int_t^T ds \int_{\R^{d}} dz \tilde{p}^{s+\varepsilon,y}(t,s, x,z) f(s,z).
$$
We define for all $f\in C^{1,2}_0([0,T)\times \R^{d},\R) $:
$$
M_{t,x}^{\varepsilon,y}f(t,x)= \int_t^Tds \int_{\R^{d}} dz L_s(\theta_{t,s+\varepsilon}(y),\nabla_x) \tilde{p}^{s+\varepsilon, y}(t,s,x,z)f(s,z).
$$
We derive from the backward Kolmogorov equation for the frozen density 
that the following equality holds:
\begin{equation}\label{rel:diff:G+M}
\partial_t G^{\varepsilon,y} f(t,x)+ M_{t,x}^{\varepsilon,y}f(t,x) = -f(t,x),\ \forall (t,x)\in [0,T)\times \R^{d}.
\end{equation}

Now, let  $h \in C^{1,2}_0([0,T)\times \R^{nd},\R)$ be an arbitrary function and define for all $(t,x)\in [0,T)\times \R^{nd} $:
\begin{eqnarray*}
\phi^{\varepsilon,y} (t,x) := \tilde{p}^{t+\varepsilon,y}(t,t+\varepsilon,x,y)h(t,y), 
\Psi_\varepsilon(t,x):= \int_{\R^{d}} dy G^{\varepsilon,y}(\phi^{\varepsilon,y})(t,x).
\end{eqnarray*}
Then, by semigroup property, we have:
\begin{eqnarray*}
\Psi_\varepsilon(t,x) &=& \int_{\R^{d}} dy \int_t^T ds \int_{\R^{d}}dz \tilde{p}^{s+\varepsilon,y}(t,s,x,z) \tilde{p}^{s+\varepsilon,y}(s,s+\varepsilon,z,y)h(s,y) \\
&=& \int_{\R^{d}}dy \int_t^T ds \tilde{p}^{s+\varepsilon,y}(t,s+\varepsilon,x,y) h(s,y).
\end{eqnarray*}

Hence, we can write:
\begin{eqnarray*}
\partial_t \Psi_\varepsilon(t,x)+ L_t(x,\nabla_x) \Psi_\varepsilon(t,x) 
= \int_{\R^d} dy \Big( \partial_t G^{\varepsilon,y} \phi^{\varepsilon,y}(t,x) + M_{t,x}^{\varepsilon,y}  \phi^{\varepsilon,y}(t,x)  \Big)\\
+\int_{\R^d}dy \Big(L_t(x,\nabla_x) G^{y}  \phi^{\varepsilon,y}(t,x) - M_{t,x}^{\varepsilon,y} \phi^{\varepsilon,y}(t,x)\Big)   \\
:= I_1^\varepsilon(t,x) +I_2^\varepsilon(t,x).
\end{eqnarray*}

Observe that  from \eqref{rel:diff:G+M}, we have:
$$
I_1^\varepsilon(t,x)=- \int_{\R^d}\tilde{p}^{t+\varepsilon,y}(t,t+\varepsilon,x,y)h(t,y)dy.
$$

Now, from Lemma \ref{convergence_dirac_temp}, when $\varepsilon \rightarrow 0$ we have the convergence:
$$
\int_{\R^d}\tilde{p}^{t+\varepsilon, y}(t,t+\varepsilon,x,y)h(t,y)dy \underset{\varepsilon \rightarrow 0}{\longrightarrow} h(t,x).
$$
Consequently, $I_1^\varepsilon(t,x)$ allows us to recover the test function $h(t,x)$ when $\varepsilon$ tends to zero, that is:
$$
\lim_{\varepsilon \rightarrow 0} \left|S^\Delta (I_1^\varepsilon)\right| = |S^\Delta h|.
$$

On the other hand, 
\begin{eqnarray*}
I_2^\varepsilon(t,x) = \int_{\R^d}dy \Big(L_t(x,\nabla_x) G^{y}  \phi^{\varepsilon,y}(t,x) - M_{t,x}^{\varepsilon,y} \phi^{\varepsilon,y}(t,x) \Big)\\
=\int_{\R^{d}}dy \int_t^T ds\Big(L_t(x,\nabla_x) - L_t(\theta_{t,s+\varepsilon}(y),\nabla_x) \Big)  \tilde{p}_\alpha^{s+\varepsilon,y}(t,s+\varepsilon,x,y) h(s,y)\\
= \int_{\R^{d}}dy \int_t^T dsH(t,s+\varepsilon,x,y) h(s,y) .
\end{eqnarray*}

From the controls of Subsection \ref{SECT_SMTH_H}, specifically, Lemma \ref{EST_H_PROP}, we have for all $(t,x)\in [0,T]\times \R^d$:
\begin{eqnarray*}
|I_2^\varepsilon (t,x)|
\le 
|h|_\infty  \int_{\R^{d}}dy \int_t^T ds |H(t,s+\varepsilon,x,y)| \le C(T+\varepsilon-t)^\omega |h|_\infty.
\end{eqnarray*}
Thus, denoting by $||S^\Delta||:=\sup_{|f|_\infty\le 1}|S^\Delta f|$, we have:
\begin{eqnarray*}
\lim_{\varepsilon \rightarrow 0} \left| S^\Delta (I_2^\varepsilon) \right| 
\le ||S^\Delta || 
\liminf_{\varepsilon \rightarrow 0} \left| I_2^\varepsilon \right|_{\infty} \le C ||S^\Delta || (T-t)^\omega |h|_\infty.
\end{eqnarray*}


Now, from \eqref{+-GENE GEL} with $f(t,x) = \Psi_\varepsilon(t,x)$, we have
$$
S^\Delta \Big( (\partial_{\cdot} + L(\cdot, \nabla_\cdot)) \Psi_\varepsilon \Big)=0 \Rightarrow |S^\Delta(I_1^\varepsilon)|=|S^\Delta(I_2^\varepsilon)|.
$$

Thus, for $T-t$ small enough,
\begin{eqnarray*}
|S^\Delta h| = \lim_{\varepsilon \rightarrow 0} \left|S^\Delta (I_1^\varepsilon)\right|
 = \lim_{\varepsilon \rightarrow 0}\left|S^\Delta I_2^\varepsilon \right| \leq1/2 \|S^\Delta\| |h|_{\infty}.
\end{eqnarray*}
By a monotone class argument, the previous inequality still holds for bounded Borel functions 
 $h$ compactly supported in $[0,T)\times \R^{d} $.
Taking the supremum over $|h|_\infty\le 1$ leads to $\|S^\Delta\| \leq 1/2 \|S^\Delta\|$. Since  $\|S^\Delta\| \leq T-t$, we deduce that $\|S^\Delta\|=0$ which proves the result on $[0,T] $. Regular conditional probabilities allow to extend the result on $\R^+ $, see e.g. Theorem 4, Chapter II, paragraph 7, in \cite{shir:96}.
\end{proof}

\subsection{Proof of Lemma \ref{LEMME_IT_KER_TEMP}.}
\label{PROOF_LEMME_IT_KER}

In Subsection \ref{EST_FRZN_DENS_SECTION}, we have obtained estimates for both the frozen density and the parametrix kernel.
In this section, we expose how these estimates are used to deduce the convergence of the parametrix series through the controls of Lemma \ref{LEMME_IT_KER_TEMP}.

\begin{lemma}\label{FIRST_STEP}
Fix $t\le \tau \le T$.
There exists $C>1$, $\omega>0$ such that for all $(x,y)\in (\R^{d})^2 $:
$$
\int_\R \bar{p}(t,\tau,x,z) \bar{H}(\tau,T,z,y) dz \le C \left( (T-\tau)^{\omega-1}+ (\tau-t)^{\omega-1}+\frac{\delta \wedge |\theta_{t,T}(y)-x|^{\eta(\alpha \wedge 1)}}{T-t} \right)\bar{p}(t,T,x,y).
$$
Consequently, when integrated in time, we have:
$$|\tilde{p}\otimes H(t,T,x,y) | \le C \Big( (T-t)^\omega\bar{p}(t,T,x,y) + \rho(t,T,x,y) \Big),$$
where we recall the notation $ \rho(t,T,x,y) = \delta \wedge |\theta_{t,T}(y)-x|^{\eta(\alpha \wedge 1)}\bar{p}(t,T,x,y)$.

\end{lemma}

\begin{proof}
Let us recall that $\bar{p}$ and $\bar{H}$ are the upper bounds for $\tilde p$ and $H$ respectively, so that:
$$
|\tilde{p}\otimes H(t,T,x,y)| \le \int_t^T d\tau \int_{\R^d} \bar{p}(t,\tau,x,z) \bar{H}(\tau,T,z,y) dz.
$$
We investigate the integration in time. We have:
\begin{eqnarray*}
\int_\R \bar{p}(t,\tau,x,z) \bar{H}(\tau,T,z,y) dz
&=& \int_{\R^d} \frac{(\tau-t)^{-d/\alpha}}{ \left(1+\frac{|x-\theta_{t,\tau}(z)|}{(\tau-t)^{1/\alpha}} \right)^{\alpha+\gamma}} \bar{q}(|x-\theta_{t,\tau}(z)|)\\ 
&&\times\frac{\delta \wedge |z-\theta_{\tau,T}(y)|^{\eta(\alpha\wedge1)}}{T-\tau}\frac{(T-\tau)^{-d/\alpha}}{ \left(1+\frac{|z-\theta_{\tau,T}(y)|}{(T-\tau)^{1/\alpha}} \right)^{\alpha+\gamma}} \bar{q}(|z-\theta_{\tau,T}(y)|) dz.
\end{eqnarray*}

Assume first that $|\theta_{t,T}(y)-x| \le C (T-t)^{1/\alpha}$.
The arguments differs according to the position of $\tau$ in $[t,T]$.
First, assume that $\tau \in [\frac{T+t}{2},T]$. We have that $\tau-t \asymp T-t$, so that 
%
$$
\bar{p}(t,\tau,x,z) \le (\tau-t)^{-d/\alpha} \asymp (T-t)^{-d/\alpha} \asymp \bar{p}(t,T,x,y).
$$

Consequently, we take $\bar{p}(\tau,x,z)$ out of the integral and use the smoothing property of Lemma \ref{SMTH_H}:
\begin{eqnarray*}
\bar{p}(t,T,x,y) \int_{\R^d} 
\frac{\delta \wedge |z-\theta_{\tau,T}(y)|^{\eta(\alpha\wedge1)}}{T-\tau}\frac{(T-\tau)^{-d/\alpha}}{ \left(1+\frac{|z-\theta_{\tau,T}(y)|}{(T-\tau)^{1/\alpha}} \right)^{\alpha+\gamma}} \bar{q}(|z-\theta_{\tau,T}(y)|) dz \\
\le C (T-\tau)^{\omega-1} \bar{p}(t,T,x,y).
\end{eqnarray*}

When, $\tau\in [0,\frac{T+t}{2}]$ we have $T-\tau \asymp T-t$, 
and we have 
\begin{eqnarray*}
\frac{1}{T-\tau}\bar{p}(\tau,T,z,y) \le C\frac{(T-\tau)^{-d/\alpha}}{T-\tau} \le C\frac{(T-t)^{-d/\alpha}}{T-t} \le C\frac{1}{T-t} \bar{p}(t,T,x,y).
\end{eqnarray*}
Next, we can bound 
$$\delta \wedge |z-\theta_{\tau,T}(y)|^{\eta(\alpha\wedge1)} \le C_T (\delta \wedge |\theta_{t,\tau}z-x|^{\eta(\alpha\wedge1)}+ \delta \wedge |x-\theta_{t,T}(y)|^{\eta(\alpha\wedge1)}).$$
Thus, we finally obtain:
\begin{eqnarray*}
\frac{1}{T-t}\bar{p}(t,T,x,y) \int_{\R^d} 
\frac{(\tau-t)^{-d/\alpha}}{ \left(1+\frac{|\theta_{t,\tau}(z)-x|}{(\tau-t)^{1/\alpha}} \right)^{\alpha+\gamma}} \bar{q}(C|\theta_{t,\tau}(z)-x|) \\
\times\Big(\delta \wedge |\theta_{t,\tau}(z)-x|^{\eta(\alpha\wedge1)}+ \delta \wedge |\theta_{t,T}(y)-x|^{\eta(\alpha\wedge1)} \Big)dz \\
\le C \Big( (\tau-t)^{\omega-1}  + \frac{\delta \wedge |\theta_{t,T}(y)-x|^{\eta(\alpha\wedge1)}}{T-t} \Big) \bar{p}(t,T,x,y).
\end{eqnarray*}

Assume now that $|\theta_{t,T}(y)-x| \ge C(T-t)^{1/\alpha}$.
In this case, the off-diagonal estimate holds for $\bar{p}(t,T,x,y)$, that is:
$$
\bar{p}(t,T,x,y) \asymp \frac{(T-t)^{1+\frac{\gamma-d}{\alpha}}}{|\theta_{t,T}(y)-x|  ^{\alpha+\gamma}} \bar{q}(|\theta_{t,T}(y)-x|  ).
$$

On the other hand, we have:
$$
 |\theta_{t,T}(y)-x| \le C_T \Big( |\theta_{t,\tau}(z)-x| +|\theta_{\tau,T}(y)-z| \Big).
$$
In other words, we have either $|\theta_{\tau,T}(y)-z|  \ge C  |\theta_{t,T}(y)-x|$, or $|\theta_{t,\tau}(z)-x|  \ge C |\theta_{t,T}(y)-x|$.
Consequently, we split $\R^{d}= D_1 \cup D_2$ with
\begin{eqnarray*}
D_1&=& \{ z \in \R^d, |\theta_{\tau,T}(y)-z|   \le  |\theta_{t,\tau}(z)-x|\}, \\
D_2&=& \{ z \in \R^d, |\theta_{\tau,T}(y)-z|   >  |\theta_{t,\tau}(z)-x|\} .
\end{eqnarray*}

Now, when $z\in D_1$, we have that $ |\theta_{t,T}(y)-x| \asymp  |\theta_{t,\tau}(z)-x|$, thus $\bar{p}(t,\tau,x,z)$ is off-diagonal and
 we can bound:
\begin{eqnarray*}
\bar{p}(t,\tau,x,z) &\le& C\frac{(\tau-t)^{1+\frac{\gamma-d}{\alpha}}}{ |\theta_{t,\tau}(z)-x|^{\alpha+\gamma}}\bar{q}( |\theta_{t,\tau}(z)-x|) \\
&\le& C \frac{(T-t)^{1+\frac{\gamma-d}{\alpha}}}{ |\theta_{t,T}(y)-x| ^{\alpha+\gamma}} \bar{q}( |\theta_{t,T}(y)-x| ) \asymp\bar{p}(t,T,x,y) .
\end{eqnarray*}
For the last inequality, we used the fact that $\bar q$ is non increasing and that $\gamma+\alpha>d$ so that the exponent in $\tau-t$ is positive.
Thus, we can take out $\bar{p}(t,\tau,x,z)$ of the integral, and use the smoothing property of $H$, Lemma \ref{SMTH_H}. Denoting by $I_{D_1}$ the integral in space in $|\tilde{p} \otimes H|$ where the space integration is over $D_1$, we have:

\begin{eqnarray*}
I_{D_1} &\le& C\bar{p}(t,T,x,y)\int_{D_1} \frac{\delta \wedge  |\theta_{\tau,t}(y)-z| ^{\eta(\alpha\wedge1)}}{T-\tau}\frac{(T-\tau)^{-d/\alpha}}{ \left(1+\frac{ |\theta_{\tau,T}(y)-z|}{(T-\tau)^{1/\alpha}} \right)^{\alpha+\gamma}} \bar{q}( |\theta_{\tau,T}(y)-z|) dz\\
&\le& C (T-\tau)^{\omega-1}\bar{p}(t,T,x,y).
\end{eqnarray*}

When $z\in D_2$, we have $ |\theta_{\tau,T}(y)-z| \asymp |\theta_{t,T}(y)-x|$.
In this case, observe that we have:
$$
\frac{1}{T-\tau}\bar{p}(T-\tau,z,y) \asymp \frac{(T-\tau)^{\frac{\gamma-d}{\alpha}}}{ |\theta_{\tau,T}(y)-z|^{\alpha +\gamma}} \bar{q}( |\theta_{\tau,T}(y)-z|) \le (T-\tau)^{\frac{\gamma-d}{\alpha}} \frac{\bar{q}( |\theta_{t,T}(y)-x|)}{ |\theta_{t,T}(y)-x|^{\alpha +\gamma}}.
$$
Thus, using Lemma \ref{SMTH_H} and recalling the constrain $\alpha+\gamma >d$, the integral becomes:
\begin{eqnarray*}
I_{D_2} &\le& C \frac{\bar{q}( |\theta_{t,T}(y)-x|)}{ |\theta_{t,T}(y)-x|^{\alpha +\gamma}} 
(T-\tau)^{\frac{\gamma-d}{\alpha}}  \int_{D_2}dz \bar{p}(t,\tau,x,z) \\
&&\times(\delta \wedge  |\theta_{\tau,t}(y)-z|^{{\eta(\alpha\wedge1)}}+ \delta \wedge  |\theta_{t,T}(y)-x|^{\eta(\alpha\wedge1)})\\
 &\le& C \frac{\bar{q}( |\theta_{t,T}(y)-x|)}{ |\theta_{t,T}(y)-x|^{\alpha +\gamma}} 
(T-\tau)^{\frac{\gamma-d}{\alpha}} \\
&&\times \left(  (t-\tau)^{\omega} +  \delta \wedge  |\theta_{t,T}(y)-x|^{\eta(\alpha\wedge1)}\right)\\
&\le& C \left(  (t-\tau)^{\omega-1} +  \frac{\delta \wedge  |\theta_{t,T}(y)-x|^{\eta(\alpha\wedge1)}}{T-t}\right)\bar{p}(t,T,x,y).
\end{eqnarray*}

Thus the proof is complete.

\end{proof}

The following Lemma controls the second step of the iterated convolutions.

\begin{lemma}\label{SECOND_STEP}
Fix $t\in (0,T]$.
There exists $C>1$, $\omega>0$ such that for all $(x,y)\in (\R^{d})^2 $:
$$
\int_\R \rho(t,\tau,x,z) \bar{H}(\tau,T,z,y) dz \le C \Big( (T-\tau)^{\omega-1}+ (\tau-t)^{\omega-1}\Big)\bar{p}(t,T,x,y).
$$
Consequently, when integrated in time, we have:
$$|\rho\otimes H(t,T,x,y) | \le C (T-t)^\omega\bar{p}(t,T,x,y).$$
\end{lemma}

\begin{proof}
The proof is similar to the previous one, but now, due to the presence of $\delta\wedge  |\theta_{t,\tau}(z)-x|^{\eta(\alpha\wedge1)}$ multiplying the first density, we do not use the triangle inequality anymore, because we are always in position to use Lemma~\ref{SMTH_H}.
\end{proof}

%
%
%

\subsection{Proof of the Lower Bound.}
\label{proof_lower_bound}
Observe first, that due to the controls on the parametrix series, the convergence of the series actually yields a diagonal lower found for the density of $(X_t)_{t> 0}$.
Indeed, we have $p(t,T,x,y) = \tilde{p}(t,T,x,y)+ p\otimes H(t,T,x,y)$.
Also, we have the upper bound $p(t,T,x,y) \le \bar{p}(t,T,x,y)$, which yields
\begin{eqnarray*}
p\otimes H(t,T,x,y) &\le& \int_t^T d\tau  \int_{\R^d} \bar{p}(t,\tau,x,z) \frac{\delta \wedge  |\theta_{\tau,T}(y)-z|^{\eta(\alpha \wedge 1)}}{T-\tau} \bar{p}(\tau,T,z,y) dz \\
&\le& \Big((T-t)^{\omega} + \delta \wedge |\theta_{t,T}(y)-x|^{\eta(\alpha \wedge 1)} \Big) \bar{p}(t,T,x,y).
\end{eqnarray*}
Thus, in diagonal regime: $ |\theta_{t,T}(y)-x|\le C(T-t)^{1/\alpha}$, we have for $t$ small enough $p(t,T,x,y) \ge C(T-t)^{-d/\alpha}$.
In other words, we have a diagonal lower bound for the density of \eqref{EDS1}.

We now turn to the off-diagonal regime.
The idea to derive a lower bound for the density is to say that 
in order to go from $x$ to $y$ in time $T-t$, we stay close to the transport of $x$ by the deterministic system, 
for a certain amount of time, 
then, a big jump brings us to a neighborhood of the pull back of $y$ by the deterministic system
and the process stays in a neighborhood of this curve.

In the off-diagonal regime: $ |\theta_{t,T}(y)-x|\ge C(T-t)^{1/\alpha}$, we write from the Chapman-Kolmogorov equation for some $t_0 \in [t,T]$:
\begin{eqnarray*}
p(t,T,x,y) &=& \int_{\R^d}dz p(t,t_0,x,z) p(t_0,Tz,y) \ge \int_{B(\theta_{t_0,T}(y),C(T-t_0)^{1/\alpha})} p(t,t_0,x,z) p(t_0,T,z,y) dz \\
&\ge& \PP\Big(X_{t_0}\in B(\theta_{t_0,T}(y),C(T-t_0)^{1/\alpha}) \Big| X_t = x \Big) \inf_{z\in B(\theta_{t_0,T}(y),C(T-t_0)^{1/\alpha})} p(t_0,T,z,y)\\
&\ge& \PP\Big(X_{t_0}\in B(\theta_{t_0,T}(y),C(T-t_0)^{1/\alpha})\Big| X_t=x\Big) C(T-t_0)^{-d/\alpha}.
\end{eqnarray*}

Consequently we have to give a lower bound for $\PP\Big(X_{t_0}\in B(\theta_{t_0,T}(y),C(T-t_0)^{1/\alpha})\Big|X_t=x\Big)$.
To this end, we introduce the process $(X_t^\delta)_{t\ge 0}$ with jumps larger than $\delta$ removed, for some $\delta$ to be specified.
Specifically, $(X_s^\delta)_{s\ge 0}$ solves the SDE:
$$
X_s^\delta = x +\int_t^s F(u,X_u^\delta)du + \int_t^s \sigma(u,X_u^\delta) dZ^\delta_u,
$$
where $(Z_u^\delta)_{u\ge 0}$ is the process $(Z_u)_{u\ge0}$ with jumps larger that $\delta$ removed. Its L\'evy measure is $\ind_{\{|z|\le \delta\}}\nu(dz)$.
Now, observe that we can recover the process $(X_s)_{s\ge 0}$ from $(X_s^\delta)_{s\ge 0}$ by introducing
the arrival times of the compound poisson process:
$$
N_s = \sum_{t<u\le s} \Delta Z_u \ind_{\{|\Delta Z_u|\ge \delta \}}.
$$
Let us denote by $(T_k)_{k\ge 1}$ the arrival times of the process $(N_s)_{s\in[t,T]}$.
We know that the variables $T_{k+1}-T_k$ are independent and have exponential distribution of parameter 
$\nu\big(B(0,\delta)^c\big)$.
Then, we have:
\begin{eqnarray*}
&&\forall t\le s\le T_1, \ X_s = X_s^\delta,\\
&&X_{T_1}= X_{T_1^-}^\delta+ \sigma(T_1,X_{T_1^-}^\delta)\Delta Z_{T_1},\\
&&\forall T_1\le s\le T_2, \   X_s=X_{T_1}+X_s^\delta-X_{T_1^-}^\delta,
\end{eqnarray*}
and so on.
We refer to the Theorem 6.2.9 in Applebaum \cite{appl:09} for a proof of this statement.
We now split:
\begin{eqnarray*}
\PP\Big(X_{t_0}\in B(\theta_{t_0,T}(y),C(T-t_0)^{1/\alpha})\Big|X_t=x\Big)&=& \PP\Big(X_{t_0}\in B(\theta_{t_0,T}(y),C(T-t_0)^{1/\alpha}); T_1 \ge t_0 \Big|X_t=x\Big)\\
&&+\PP\Big(X_{t_0}\in B(\theta_{t_0,T}(y),C(T-t_0)^{1/\alpha}); T_1 \le t_0\Big|X_t=x\Big).
\end{eqnarray*}

Using the Markovian notations $\PP^{t,x}(\cdot) = \PP(\cdot|X_t=x)$, we thus focus on:
\begin{eqnarray*}
&&\PP^{t,x}\Big(X_{t_0}\in B\big(\theta_{t_0,t}(y),C(T-t_0)^{1/\alpha}\big); T_1 \le t_0\Big) \\
&=&\E^{t,x} \left[ \PP^{t,x} \Big(X_{t_0}\in B\big(\theta_{t_0,T}(y),C(T-t_0)^{1/\alpha}\big) \Big| \mathcal{F}_{T_1} \Big) \ind_{\{T_1 \le t_0 \}}   \right],
\end{eqnarray*}
where we denoted $\mathcal{F}_{T_1} =\sigma(X_s^\delta; s \le T_1)$, the filtration generated by $X_s^\delta$ until time $T_1$. 
Now, by the strong Markov property, we have that 
\begin{eqnarray*}
\PP^{t,x} \Big(X_{t_0}\in B(\theta_{t_0,T}(y),C(T-t)^{1/\alpha}) \Big| \mathcal{F}_{T_1} \Big)&=&\PP^{T_1,X_{T_1}}\Big( X_{t_0} \in B(\theta_{t_0,T}(y),C(T-t_0)^{1/\alpha}) \Big)\\
& =& \int_{B(\theta_{t_0,T}(y),C(T-t_0)^{1/\alpha})} p(T_1,t_0, X_{T_1},z)dz.
\end{eqnarray*}
Thus:
\begin{eqnarray*}
&&\PP^{t,x}\Big(X_{t_0}\in B(\theta_{t_0,T}(y),C(T-t)^{1/\alpha}); T_1 \le t_0\Big) \\
&=&\E^{t,x} \left[\int_{B(\theta_{t_0,T}(y),C(T-t_0)^{1/\alpha})} p(T_1,t_0, X_{T_1},z)dz\ind_{\{T_1 \le t_0 \}}   \right].
\end{eqnarray*}

Now, since $X_{T_1}= X_{T_1^-}^\delta+ \sigma(T_1,X_{T_1^-}^\delta)\Delta Z_{T_1}$, 
and since $T_1$ is the first jump larger that $\delta$, conditionally to $X_{T_1^-}$ we have that $\sigma(T_1, X_{T_1^-}) \Delta Z_{T_1}+X_{T_1^-}$ is a Poisson process on $\R^d\backslash B(0,\delta)$.
Thus, we have for all test function $f$, given $X_{T_1^-}$, the law of $X_{T_1}$ is: 
\begin{eqnarray*}
\E[f(X_{T_1}) | X_{T_1^-}] &=& \E [ f(\sigma(T_1, X_{T_1^-}) \Delta Z_{T_1} + X_{T_1^-} )| X_{T_1^-}] \\
&=& \int_{ \{|w|\ge \delta\}} f \big( \sigma(T_1,X_{T_1^-})w + X_{T_1^-} \big) \frac{ \nu(dw)}{\nu\big(B(0,\delta)^c\big)} .
\end{eqnarray*}

Consequently, we obtain:
\begin{eqnarray*}
&& \E^{t,x} \left(\int_{B(\theta_{t_0,T}(y),C(T-t_0)^{1/\alpha})} p(T_1,t_0, X_{T_1},z)dz \bigg| X_{T_1^-}\right) \\
&& =\int_{\R^d}  
 \int_{B(\theta_{t_0,T}(y),C(T-t_0)^{1/\alpha})}dz\  p(T_1, t_0,\sigma(T_1,X_{T_1^-}^\delta)w + X_{T_1^-}\delta ,z)
\frac{ \nu(dw)}{\nu\big(B(0,\delta)^c\big)} .
\end{eqnarray*}

Now, we exploit the fact that $T_1$ in independent and exponentially distributed with parameter $\nu(B(0,\delta)^c)$ to write:
\begin{eqnarray*}
\PP^{t,x}\Big(X_{t_0}\in B(\theta_{t_0,T}(y),C(T-t)^{1/\alpha}); T_1 \le t_0\Big)
 &=&\E^{t,x} \bigg[ 
\int_t^{t_0} ds\int_{B(\theta_{t_0,T}(y),C(T-t_0)^{1/\alpha})}dz
\int_{\R^d}\frac{ \nu(dw)}{\nu\big(B(0,\delta)^c\big)}\\
&& \times p(s,t_0,\sigma(s,X_{s}^\delta)w + X_{s}^\delta ,z)
\nu\big(B(0,\delta)^c\big) e^{-s \nu\big(B(0,\delta)^c\big)} 
 \bigg].
\end{eqnarray*}

Observe that the quantity $\nu\big(B(0,\delta)^c\big)$ gets cancelled.
Now, we can give a lower bound  by localizing the integral over $w$ so that $\sigma(s,X_{s}^\delta)w + X_{s}^\delta$ is close to $\theta_{s,t_0}(z)$. That is, where the density $p(s,t_0, \sigma(X_{s}^\delta)w + X_{s}^\delta ,z)$ is in diagonal regime:

\begin{eqnarray*}
&&\PP^{t,x}\Big(X_{t_0}\in B(\theta_{t_0,T}(y),C(T-t_0)^{1/\alpha}); T_1 \le t_0\Big)  \\
&\ge&\E^{t,x} \bigg[ 
\int_t^{t_0} ds \int_{B(\theta_{t_0,T}(y),C(T-t_0)^{\frac1\alpha})} dz
\int_{ \{ |\sigma(s,X_{s}^\delta)w + X_{s}^\delta-\theta_{s,t_0}(z)| \le C(t_0-s)^{\frac1 \alpha} \}} \nu(dw) \\
&& \times p\Big(s,t_0, \sigma(s,X_{s}^\delta)w + X_{s}^\delta ,z\Big)
 e^{-s \nu(B(0,\delta)^c)}
 \bigg]\\
& \ge& \E^{t,x} \bigg[ \int_t^{t_0} ds \Big(t_0-s\Big)^{-d/\alpha}
\int_{B(\theta_{t_0,T}(y),C(T-t_0)^{1/\alpha})} dz \\
&& \times \nu \bigg(B \Big( \sigma(s,X_s^\delta)^{-1}(\theta_{s,t_0}(z) - X_{s}^\delta),C\left(t_0-s\right)^{1/\alpha} \Big) \bigg)
 e^{-s \nu(B(0,\delta)^c)}
 \bigg].
\end{eqnarray*}
Additionally, we can lower bound the last probability by localizing $X_s^\delta$ close to $\theta_{s,t}(x)$:
\begin{eqnarray*}
&&\PP^{t,x}\Big(X_{t_0}\in B(\theta_{t_0,T}(y),C(T-t)^{1/\alpha}); T_1 \le t_0\Big)\\
&\ge&\E^{t,x} \bigg[ \int_t^{t_0} ds\ind_{\{|X_s^\delta - \theta_{s,t}(x) | \le C(s-t)^{1/\alpha} \}} \Big(t_0-s\Big)^{-d/\alpha}
\int_{B(\theta_{t_0,t}(y),C(T-t_0)^{1/\alpha})} dz \\
&& \times \nu \bigg(B \Big( \sigma(s,X_s^\delta)^{-1}(\theta_{s,t_0}(z) - X_{s}^\delta),C\left(t_0-s\right)^{1/\alpha} \Big) \bigg)
 e^{-s \nu(B(0,\delta)^c)}
\bigg].
\end{eqnarray*}

Now, from assumption \textbf{[H-LB]}, $ \nu(B(0,\delta)^c)\le 1/\delta^\alpha$ so that
taking $\delta = (T-t)^{1/\alpha}$ yields $ e^{-s \nu(B(0,\delta)^c)} \ge C$.
Also, since $z \in B(\theta_{t_0,T}(y),C(T-t_0)^{\frac1\alpha})$, by the Lipschitz property of the flow, 
$$
\theta_{s,t_0}(z) \in \theta_{s,t_0} \Big(B(\theta_{t_0,T}(y),C(T-t_0)^{\frac1\alpha})\Big) \subset B(\theta_{s,T}(y),C (T-t_0)^{\frac1\alpha}),
$$ 
up to a modification of $C$ for the last inclusion.
On the other hand, $X_{s}^\delta\in B( \theta_{s,t}(x),s^{1/\alpha})$, 
thus, 
$$
\sigma(s,X_s^\delta)^{-1}(\theta_{s,t_0}(z) - X_{s}^\delta) \in B\big(\sigma(s,X_s^\delta)^{-1}( \theta_{t,T}(y) - x), C(T-t)^{1/\alpha} \big).
$$

Now, from \textbf{[H-LB]}, $A_{low}$ is invariant under the action of $\sigma(t,x)$ for all $(t,x)\in \R_+ \times \R^d$, thus, if
\begin{equation}\label{inclusion}
B\big( \theta_{t,T}(y) - x, C(T-t)^{1/\alpha} \big) \subset A_{low},
\end{equation}
%
%
we can use the lower bound in \textbf{[H-LB]} to get:
\begin{eqnarray*}
\nu \left(B \Big( \sigma(s,X_s^\delta)^{-1}(\theta_{s,t_0}(z) - X_{s}^\delta),C\left(t_0-s\right)^{1/\alpha} \Big) \right) 
&\ge& C\left(t_0-s\right)^{\gamma/\alpha} \frac{ \underline{q}(  |\sigma(s,X_s^\delta)^{-1}(\theta_{s,t_0}(z) - X_{s}^\delta)|  ) }{|\sigma(s,X_s^\delta)^{-1}(\theta_{s,t_0}(z) - X_{s}^\delta)|^{\gamma+\alpha}}.
\end{eqnarray*}

Observe that this is exactly the condition \eqref{inclusion} of Theorem \ref{MTHM_TEMP}.
We thus obtain:
\begin{eqnarray*}
&&\PP^{t,x}\Big(X_{t_0}\in B(\theta_{t_0,T}(y),C(T-t_0)^{1/\alpha}); T_1 \le t_0\Big)\\
& \ge& C \E^{t,x} \bigg[ 
\int_t^{t_0} ds \ind_{\{|X_s^\delta - \theta_{s,t}(x) | \le C(s-t)^{1/\alpha} \}}
\left(t_0-s\right)^{\frac{\gamma-d}{\alpha}}\\
&&\times \int_{B(\theta_{t_0,T}(y),C(T-t_0)^{1/\alpha})} dz \frac{ \underline{q}(  |\sigma(s,X_s^\delta)^{-1}(\theta_{t,s}(z) - X_{s}^\delta)|  ) }{|\sigma(s,X_s^\delta)^{-1}(\theta_{t,s}(z) - X_{s}^\delta)|^{\gamma+\alpha}}
 \bigg].
\end{eqnarray*}

Consequently, since the function $u \mapsto \underline{q}(u)|u|^{-\gamma-\alpha}$ is decreasing, the lower bound will follow from the upper bound:

$$
|\sigma(s,X_s^\delta)^{-1}(\theta_{s,t_0}(z) - X_{s}^\delta)| \le C |y-\theta_{T,t}(x)|.
$$

We write from the ellipticity of $\sigma$:
\begin{eqnarray*}
|\sigma(s,X_s^\delta)^{-1}(\theta_{s,t_0}(z) - X_{s}^\delta)| &\le & C |\theta_{s,t_0}(z) - X_{s}^\delta|\\
&\le& C(|\theta_{s,t_0}(z) -\theta_{s,t}(x)| +|\theta_{s,t}(x) - X_{s}^\delta|).
\end{eqnarray*}
Now, in the considered set, $|\theta_{s,t}(x) - X_{s}^\delta|\le C(s-t)^{1/\alpha}\le C(T-t)^{1/\alpha} \le  C|\theta_{t,T}(y) -x|$.
Thus, we have:
$$
|\sigma(s,X_s^\delta)^{-1}(\theta_{s,t_0}(z) - X_{s}^\delta)| \le C \big( |\theta_{s,t_0}(z) -\theta_{s,t}(x)|+C|\theta_{t,T}(y) -x|\big).
$$

On the other hand, we can write:
\begin{eqnarray*}
|\theta_{s,t_0}(z) -\theta_{s,t}(x)| &\le& |\theta_{s,t_0}(z)-\theta_{s,T}(y)| + |\theta_{s,T}(y) -\theta_{s,t}(x)|.
\end{eqnarray*}
Thus, from the Lipschitz property of the flow, 
$$
 |\theta_{s,T}(y) -\theta_{s,t}(x)| \le C_T  |\theta_{t,T}(y) -x|.
$$
On the other hand, we have $\theta_{s,T}(y) = \theta_{s,t_0}\circ\theta_{t_0,T}(y)$ so that:
$$
|\theta_{s,t_0}(z)-\theta_{s,T}(y)| = |\theta_{s,t_0}(z)-\theta_{s,t_0}\circ\theta_{t_0,T}(y)| \le C_T  |z-\theta_{t_0,T}(y)|,
$$
where to the get the last inequality, we once again relied on the Lipschitz property of the flow.
We recall that $|z-\theta_{t_0,T}(y)| \le C(T-t_0)^{1/\alpha} \le |\theta_{t,T}(y) -x|$, consequently we finally obtain:

$$
|\sigma(s,X_s^\delta)^{-1}(\theta_{s,t_0}(z) - X_{s}^\delta)|  \le C_T  |\theta_{t,T}(y) -x|.
$$

Using this last inequality to estimate the probability:

\begin{eqnarray*}
&&\PP^{t,x}\Big(X_{t_0}\in B(\theta_{t,t_0},C(T-t)^{1/\alpha}); T_1 \le t_0\Big)\\
  &\ge& 
C \E^{t,x} \left[ 
\int_t^{t_0} ds
\ind_{\{|X_s^\delta - \theta_{s,t}(x) | \le (s-t)^{1/\alpha} \}}
 \left(t_0-s\right)^{\frac{\gamma-d}{\alpha}}
\int_{B(\theta_{t_0,T}(y),C(T-t_0)^{1/\alpha})} dz \frac{ \underline{q}(   |\theta_{t,T}(y) -x| ) }{ |\theta_{t,T}(y) -x|^{\gamma+\alpha}}
 \right]\\
& \ge& C(T-t_0)^{d/\alpha} \frac{ \underline{q}(   |\theta_{t,T}(y) -x| ) }{ |\theta_{t,T}(y) -x|^{\gamma+\alpha}} 
\int_t^{t_0} ds \left(t_0-s\right)^{\frac{\gamma-d}{\alpha}} \PP (|X_s^\delta - \theta_{s,t}(x) | \le (s-t)^{1/\alpha}),
\end{eqnarray*}
where $(T-t_0)^{d/\alpha}$ comes from the volume of the ball $B(\theta_{t_0,T}(y),C(T-t_0)^{1/\alpha})$ obtained from the integral in $dz$.
Using the diagonal lower estimates for the density, we actually see that $\PP (|X_s^\delta - \theta_{s,t}(x) | \le (s-t)^{1/\alpha}) \asymp 1$, therefore:
\begin{eqnarray*}
\PP^{t,x}\Big(X_{t_0}\in B(\theta_{t_0,T}(y),C(T-t_0)^{1/\alpha}); T_1 \le t_0 \Big)
 \ge C(T-t_0)^{d/\alpha} (t_0-t)^{1+\frac{\gamma-d}{\alpha}} \frac{ \underline{q}(   |\theta_{t,T}(y) -x| ) }{ |\theta_{t,T}(y) -x|^{\gamma+\alpha}} .
\end{eqnarray*}

Returning to the first estimate on the density yields:
\begin{eqnarray*}
p(t,T,x,y) &\ge& C(T-t_0)^{-d/\alpha}\PP^{t,x}\Big(X_{t_0}\in B(\theta_{t_0,T}(y),C(T-t_0)^{1/\alpha}); T_1 \le t_0\Big)\\
&\ge& 
C(T-t_0)^{-d/\alpha}(T-t_0)^{d/\alpha} (t_0-t)^{1+\frac{\gamma-d}{\alpha}} \frac{ \underline{q}(   |\theta_{t,T}(y) -x| ) }{ |\theta_{t,T}(y) -x|^{\gamma+\alpha}}\\
&=&C(t_0-t)^{1+\frac{\gamma-d}{\alpha}} \frac{ \underline{q}(   |\theta_{t,T}(y) -x| ) }{ |\theta_{t,T}(y) -x|^{\gamma+\alpha}}.
\end{eqnarray*}
Finally, to get the announced bound, we see that we have to choose $t_0$ such that $t_0-t\asymp T-t$.
This gives:
$$
p(t,T,x,y) \ge C(T-t)^{1+\frac{\gamma-d}{\alpha}} \frac{ \underline{q}(   |\theta_{t,T}(y) -x| ) }{ |\theta_{t,T}(y) -x|^{\gamma+\alpha}}.
$$
which is the off-diagonal lower bound announced.

\begin{remark}
We point out that the assumption \textbf{[H-LB]} appears quite naturally in this procedure as it serves here to give a lower bound on the $\nu$-measure of balls.
Also, we see the that $\sigma(t,x)$ preserves $A_{low}$ is important so that we can use the lower bound for the $\nu$-measure of balls.
\end{remark}

\section*{Acknowledgments}
This article was prepared within the framework of a subsidy granted to the HSE by the Government of the Russian Federation for the implementation if the Global Competitiveness Program.

\bibliographystyle{alpha}
\bibliography{MyLibrary.bib}

%
%
%
%
%
%
%



\end{document}